\documentclass[final,leqno]{siamltex}

\usepackage{amssymb}                                %
\usepackage{amsmath}                                %
\usepackage{amsfonts}                               %
\usepackage{graphicx}                               %
\usepackage{graphics}                               %
\usepackage{color}                                  %
\usepackage{rotating}              %
\usepackage[none]{hyphenat}
\def\bfm#1{\protect{\makebox{\boldmath $#1$}}}
\def\eep {\bfm{\varepsilon}}
\def\b {\bfm{b}}
\def\u {\bfm{u}}
\def\f {\bfm{f}}
\def\g {\bfm{g}}

\newcommand{\R}{\mathbb R}
\newcommand{\N}{\mathbb N}
\newcommand{\D}{\mathcal D}

\newcommand{\wt}{\widetilde}
\newcommand{\wh}{\widehat}
\newcommand{\Gj}{\Gamma_{\hspace{-.07cm}j}}

\newcommand{\argmin}{\operatornamewithlimits{argmin}}

\title{The Direct Radial Basis Function Partition of Unity (D-RBF-PU) Method for Solving PDEs}

\author{Davoud Mirzaei\thanks{Department of Applied Mathematics and Computer Science, Faculty of Mathematics and Statistics, University of Isfahan, 81746-73441 Isfahan, Iran. (\texttt{d.mirzaei@sci.ui.ac.ir})
\newline
\textbf{Funding:} The work of the author was supported by the Iranian National Science Foundation (INSF) under grant No. 98012657.
\newline
{Date: September 23, 2019 (submitted), October 26, 2020 (accepted)}
}
}
\begin{document}
\maketitle

\begin{abstract}
In this paper, a new localized radial basis function (RBF) method based on partition of unity (PU) is proposed for solving boundary and initial-boundary value problems.
The new method is benefited from a direct discretization approach and is called the `direct RBF partition of unity (D-RBF-PU)' method.
Thanks to avoiding all derivatives of PU weight functions as well as all lower derivatives of local approximants, the new method is faster and simpler than the standard RBF-PU method. Besides, the discontinuous PU weight functions can now be utilized to develop the method in a more efficient and less expensive way. 
Alternatively, the new method is an RBF-generated finite difference (RBF-FD) method in a PU setting which is much faster and in some situations more accurate than the original RBF-FD. The polyharmonic splines are used for local approximations, and the error and stability issues are considered. Some numerical experiments
on irregular 2D and 3D domains, as well as cost comparison tests, are performed
to support the theoretical analysis and to show the efficiency of the new method.
\end{abstract}



\pagestyle{myheadings}
\thispagestyle{plain}
\markboth{D-RBF-PU Method for Solving PDEs}{D. Mirzaei}

\section{Introduction}\label{sect_introduction}
Approximation by {\em kernels} and in particular by {\em radial basis functions (RBFs)} has received a lot of attention due to
many attractive advantages such as ease of implementation, flexibility with respect to geometry and dimension and giving spectral accuracy in some situations.
However,
the global RBF approximations produce full and ill-conditioned matrices which make them restricted for large scale problems.
So, localized approaches, such as {\em RBF-generated finite difference (RBF-FD)} and {\em RBF partition of unity (RBF-PU)}  methods, are currently being developed.

The earliest reference to RBF-FD seems to be a conference presentation
in 2000 \cite{tolstykh:2000-1}. Then, this method was developed in three
simultaneous studies \cite{shu-et-al:2003-1,tolstykh:2003-1,wright:2003-thesis} in 2003.
As in the classical FD methods, RBF-FD results in sparse matrices with an additional advantage that has all the flexibility of global RBFs in terms of handling irregular geometries and scattered node layouts.
To avoid the ill-conditioning at the {\em near flat} cases, i.e. for very small values of RBF shape parameters, some technical algorithms have been
introduced in \cite{fornberg-et-al:2011-1,fornberg-et-al:2013-1,larsson-et-al:2013-1} for Gaussian RBF and in \cite{bozzini-et-al:2015-1,fornberg-piret:2007-1,fornberg-wright:2004-1,wright-fornberg:2017-1} for all types of RBFs. Equipping with such algorithms, the RBF-FD method has been successfully applied on a large class of PDEs in Euclidian spaces and on smooth sub-manifolds \cite{flyer-et-al:2012-1, fornberg-lehto:2011-1,fornberg-flyer:2015-1}.

Although a PU method has been introduced by Shepard in 1968 \cite{shepard:1968-1}, the first combination with RBF interpolation goes back
to \cite{wendland:2002-3} in 2002. However, a PU finite element method for solving PDEs has been proposed in
\cite{babuska-melenk:1997-1, melenk-babuska:1996-1} few years earlier.
The RBF-PU collocation method for solving transport equations on the unit sphere has been
developed in \cite{aiton:2014-thesis}.
The capability of the RBF-PU method for numerical solution of
parabolic PDEs in financial mathematics has been investigated in \cite{Mollapourasl-et-al:2019-1,safdari-et-al:2015-1,shcherbakov:2016-1,shcherbakov-larsson:2016}.
Preconditioning schemes are studied in \cite{heryudono-et-al:2016-1} and a least square RBF-PU method is proposed in \cite{larsson-et-al:2017-1}. Adaptivity and stability issues via variably scaled kernels
are recently given in \cite{DeMarchi-et-al:2019-1}. Other applications can be found in \cite{cavoretto:2015-1,DeMarchi-et-al:2019-2,DeRossi-et-al:2016-1}.

In this paper we introduce a new RBF-PU method for solving boundary value problems.\ We use
the idea of {\em direct discretization} and link the RBF-PU to the RBF-FD and construct a
{\em direct RBF-PU} method which is more efficient than both RBF-FD and RBF-PU methods. {Indeed, the classical FD method as well as the RBF-FD method use the direct approach
for discretizing a PDE operator to a finite dimensional differentiation matrix. The direct approach has been also used in \cite{mirzaei-et-al:2012-1} to speed up the computations of moving least squares derivatives and then
in \cite{mirzaei-schaback:2013-1} to accelerate the meshless local Petrov-Galerkin method.
We refer the reader to \cite{schaback:2013-1} for more details
about the direct discretization methods.
}

The rest of this paper is organized as follows.
In section \ref{sect-PUmethods}, the idea of PU in approximation theory is reviewed. In section \ref{sect-RBF-PU}, the combination of RBF approximation with PU weights and the classical RBF-PU method for solving PDEs are presented. In section \ref{sect-RBF-FD}, the well-known RBF-FD method which is in connection with the new method is briefly reviewed. In section \ref{sect-RBF-FD-PU}, the new direct RBF-PU method is introduced and its connections to some variations of the RBF-FD method are derived. Also, the scaling property of polyharmonic kernels and a stabilization technique based on scaling are recalled.  In section \ref{sect-error-stability}, the theoretical foundation of the method is provided and the consistency and stability issues are considered. Finally, in section \ref{sect-numerical-results}, some
numerical experiments and comparisons with other RBF-based methods are given.

As a remark on notation, we will use bold math symbols for vectors as far as we are writing in numerical linear algebra. Matrices are also denoted by
capital non-bold symbols.

\section{Partition of unity}\label{sect-PUmethods}
Let $\{\Omega_\ell\}_{\ell=1}^{N_c}$
be an open and bounded
covering of $\Omega$ that means all $\Omega_\ell$ are open and bounded and $\Omega\subset \bigcup_{\ell=1}^{N_c}\Omega_\ell$. A family of nonnegative
functions $\{w_\ell\}_{\ell=1}^{N_c}$ is called a partition of unity (PU) with respect to the covering
$\{\Omega_\ell\}$ if
\begin{itemize}
\item[(1)] $\mathrm{supp}(w_\ell) \subseteq \Omega_\ell$,
\item[(2)]$\displaystyle\sum_{\ell=1}^{N_c}w_\ell(x)=1,\; \forall x\in\Omega$.
\end{itemize}
{Sometimes, for a differentiation purpose, one needs to impose more regularity on PU weight functions $w_\ell$,
and may assume $w_\ell\in C^k(\R^d)$ and for every $\alpha\in \N_0^d$ with $|\alpha|\leqslant k$ there exists a constant
$C_\alpha$ such that
\begin{equation}\label{Dw_ell-bound}
\|D^\alpha w_\ell \|_{L_\infty(\Omega_\ell)} \leqslant C_\alpha \rho_\ell^{-|\alpha|},
\end{equation}
where $\rho_\ell = \frac{1}{2}\sup_{x,y\in\Omega_\ell}\|x-y\|_2$. In this case, $\{w_\ell\}$ is called a $k$-stable PU with respect to $\{\Omega_\ell\}$ \cite{wendland:2005-1}.
}

We start with an overlapping covering $\{\Omega_\ell\}_{\ell=1}^{N_c}$ of $\Omega$. If we assume
$V_\ell$ is an approximation space on $\Omega_\ell$ and $s_\ell\in V_\ell$ is a local approximant of a function $u$ on $\Omega_\ell$, then
\begin{equation}\label{pu-approx}
s = \sum_{\ell=1}^{N_c}w_\ell s_\ell
\end{equation}
is a global approximation of $u$ on $\Omega$ which is formed by joining the local approximants $s_\ell$ via PU weights $w_\ell$.
For example, if $X=\{x_1,\ldots,x_N\}\subset \Omega$, $X_\ell= X\cap \Omega_\ell$ and $s_\ell$ are $u$ interpolants on $X_\ell$ then
we can simply show that $s$ is a $u$ interpolant on $X$.

A possible choice for $w_\ell$ is the Shepard's weights
\begin{equation}\label{w-shepardform}
w_\ell(x)= \frac{\psi_\ell(x)}{\sum_{j=1}^{N_c}\psi_j(x)},\quad 1\leqslant \ell\leqslant N_c,
\end{equation}
where $\psi_\ell$ are nonnegative, nonvanishing and compactly supported functions on $\Omega_\ell$.

If $w_\ell$ and $s_\ell$ are smooth enough then the PU approximation \eqref{pu-approx} can be used for solving differential equations.
To see some applications see \cite{babuska-melenk:1997-1,melenk-babuska:1996-1} that apply the PU method on finite element spaces,
and \cite{aiton:2014-thesis,cavoretto:2015-1,DeMarchi-et-al:2019-2,DeMarchi-et-al:2019-1,DeRossi-et-al:2016-1,Garmanjani-et-al:2018-1,larsson-et-al:2017-1,
Mollapourasl-et-al:2019-1,safdari-et-al:2015-1,shcherbakov:2016-1,shcherbakov-larsson:2016} which combine the PU method with RBF approximations.

To describe the overall approach, assume that we are looking for the approximate solution of a PDE problem of the form
\begin{align}
L u &= f,\;\;  \mbox{in }\Omega, \label{Lu=f}\\
B u &= g,\;\;  \mbox{on }\Gamma \label{Bu=g}
\end{align}
where $\Omega$ is a domain in $\R^d$, $\Gamma=\partial \Omega$ denotes its boundary and $L$ and $B$ are linear differential operators
defined and continuous on some normed linear space $U$ in which the true solution of \eqref{Lu=f}-\eqref{Bu=g} should lie.
Here $B$ is the boundary operator describing the Drichlet and/or Neumann boundary conditions.
To obtain a numerical solution,
the PDE operators $L$ and $B$ should operate on $s$ (and hence on products $w_\ell s_\ell$) in \eqref{pu-approx} to get
\begin{equation*}
Lu \approx Ls = \sum_{\ell=1}^{N_c} L(w_\ell s_\ell),\quad Bu \approx Bs = \sum_{\ell=1}^{N_c} B(w_\ell s_\ell),
\end{equation*}
where $s_\ell$ is local approximation of $u$ in patch $\Omega_\ell$.
Differential operators $L$ and $B$ should contain certain partial derivatives $D^\alpha$ for multi-indices $\alpha\in \N_0^d$.
Using the Leibniz's rule we have
$$
D^\alpha s = \sum_{\ell=1}^{N_c} \sum_{|\beta|\leqslant |\alpha|}{\beta \choose \alpha}D^\beta w_\ell D^{\alpha-\beta}s_\ell,
$$
provided that both $w_\ell$ and $s_\ell$ are smooth enough.
For example if $L=\Delta=D^{(2,0,\ldots,0)}+D^{(0,2,\ldots,0)}+\cdots+D^{(0,0,\ldots,2)}$, the well-known Laplacian operator, then
$$
\Delta s = \sum_{\ell=1}^{N_c} \left(s_\ell \Delta w_\ell + 2\nabla w_\ell \cdot\nabla s_\ell + w_\ell \Delta s_\ell\right),
$$
where derivatives of $w_\ell$ are even more complicated if $w_\ell$ is defined as \eqref{w-shepardform}.
This will also increase the computational costs of the method.
This paper proposes an alternative approach that avoids the above computations and reduces both computational cost and algorithmic complexity.

\section{RBF-PU method}\label{sect-RBF-PU}
Let $\phi:\R^d\to\R$ be a radial and conditionally positive definite function of order $n$ and $X=\{x_1,x_2,\ldots,x_N\}$ be a set of trial points distributed
in $\Omega\subset\R^d$, the domain in which the PDE is posed. Let $\{\Omega_\ell\}_{\ell=1}^{N_c}$ be an open and bounded covering of $\Omega$ and
$X_\ell=X\cap \Omega_\ell$, $1\leqslant \ell\leqslant N_c$, be sets of trial points in patches $\Omega_\ell$.
Assume further that
$$J_\ell:=\{j\in\{1,\ldots,N\}:x_j\in X_\ell\}.$$
Local RBF approximation spaces in $\Omega_\ell$ are defined as
\begin{equation*}
V_\ell =V_{\phi,X_\ell}:=\mathrm{span}\{\phi(\cdot-x_j): j\in J_\ell\} \oplus \mathbb P_{n-1}(\R^d), \quad 1\leqslant \ell\leqslant N_c,
\end{equation*}
and local approximants $s_\ell\in V_\ell$ of function $u$ are expressed as
\begin{equation*}
s_\ell(x) = \sum_{j\in J_\ell}c_j\phi(x-x_j) + \sum_{k=1}^Q b_k p_k(x),\quad x\in\Omega_\ell\cap\Omega,
\end{equation*}
where $\{p_1,\ldots,p_Q\}$ is a basis for polynomial space $\mathbb P_{n-1}(\R^d)$ and $Q=\frac{(n+d-1)!}{d!(n-1)!}$ is its dimension.
For an interpolation problem,
coefficient vectors $c$ and $b$ are obtained by enforcing interpolation conditions $s_\ell(x_j) = u(x_j)$ for $j\in J_\ell$ together with
side conditions
$$
\sum_{j\in J_\ell}c_j p_k(x_j) = 0,\quad j\in J_\ell,\; 1\leqslant k\leqslant Q ,
$$
leading to linear system of equations
\begin{equation*}
\begin{bmatrix}
\Phi & P\\ P^T & 0
\end{bmatrix}
\begin{bmatrix}
c \\ b
\end{bmatrix}=
\begin{bmatrix}
u|_{X_\ell} \\ 0
\end{bmatrix},
\end{equation*}
where $\Phi(i,j)=\phi(x_j-x_i)$, $i,j\in J_\ell$, $P(j,k)=p_k(x_j)$, $j\in J_\ell$, $1\leqslant k\leqslant Q$, and
$u|_{X_\ell}= (u(x_j))$, $j\in J_\ell$.
This system is uniquely solvable if and only if $X_\ell$ is a $\mathbb P_{n-1}(\R^d)$-unisolvent set, meaning that the only polynomial from $\mathbb P_{n-1}(\R^d)$ which is zero on $X_\ell$ is the zero function.

The interpolant $s_\ell$ can also be written in the Lagrange form as
\begin{equation}\label{sl-lagrangeform}
s_\ell(x) = \sum_{j\in J_\ell} u_j^*(\ell;x) u(x_j),
\end{equation}
where Lagrange functions $u_j^*(\ell;\cdot)$ are solution of
\begin{equation*}
\begin{bmatrix}
\Phi & P\\ P^T & 0
\end{bmatrix}
\begin{bmatrix}
u^*(\ell;x) \\ v^*(\ell;x)
\end{bmatrix}=
\begin{bmatrix}
\phi(\cdot-x)|_{X_\ell} \\ p(x)
\end{bmatrix},
\end{equation*}
for $p(x)=[p_1(x),\ldots,p_Q(x)]^T$, and satisfy $u_j^*(\ell;x_i)=\delta_{ij}$.
Substituting \eqref{sl-lagrangeform} into \eqref{pu-approx} yields
\begin{equation}\label{s-lagrangeform}
s(x) = \sum_{\ell=1}^{N_c}\sum_{j\in J_\ell}\Big(w_\ell(x) u_j^*(\ell;x) \Big)u(x_j), \quad x\in\Omega.
\end{equation}
For a fixed $x\in\Omega$ since $w_\ell(x)=0$ if $x\notin \Omega_\ell$, we may introduce the index family
$$
I(x):= \big\{\ell\in\{1,2,\ldots,N_c\}: x\in \Omega_\ell\big\}
$$
to replace the summation script on $\ell$ in \eqref{s-lagrangeform} by $\ell\in I(x)$.

Now, we review the known {\em collocation} RBF-PU method for solving the PDE problem \eqref{Lu=f}-\eqref{Bu=g}.
Assume
$$Y=\{y_1,\ldots,y_M\}$$
is a set of test points in $\Omega$ and on its boundary $\Gamma$.
This set may be different from the trial set of points $X$ but for some theoretical reasons
{(to guarantee the solvability of the final unsymmetric system)}
we may assume $M\geqslant N$.
Assume $Y=Y_\Omega \cup Y_\Gamma$ where $Y_\Omega$ and $Y_\Gamma$ contain interior and boundary test points, respectively, and $Y_\Omega \cap Y_\Gamma=\emptyset$.
In a collocation method the PDE and its boundary conditions are sampled at sets of points $Y_\Omega$ and $Y_\Gamma$, respectively, to get
\begin{align}
(Lu)(y_k) =& f(y_k), \quad y_k\in Y_\Omega, \label{Luk=fk}\\
(Bu)(y_k) =&\, g(y_k), \quad y_k\in Y_\Gamma \label{Buk=gk}
\end{align}
which is a semi-discrete form of \eqref{Lu=f}-\eqref{Bu=g}.

The standard RBF-PU method,
uses the approach given at the end of section \ref{sect-PUmethods} where
$Lu$ and $Bu$ are approximated by $Ls$ and $Bs$, respectively, i.e.,
\begin{equation*}
Lu\approx Ls = \sum_{\ell=1}^{N_c}\sum_{j\in J_\ell}L\Big(w_\ell\, u_j^*(\ell;\cdot) \Big)u(x_j), \quad
Bu\approx Bs = \sum_{\ell=1}^{N_c}\sum_{j\in J_\ell}B\Big(w_\ell\, u_j^*(\ell;\cdot) \Big)u(x_j).
\end{equation*}
Inserting in \eqref{Luk=fk}-\eqref{Buk=gk} and replacing `$\approx$' by `$=$' lead to
 $M\times N$ and unsymmetric linear system of equations
\begin{equation}\label{final-system}
\begin{bmatrix}
A_{L} \\ A_B
\end{bmatrix}
\u
=
\begin{bmatrix}
f|_{Y_\Omega}\\g|_{Y_\Gamma}
\end{bmatrix},
\end{equation}
where $\u=[u_1,\ldots,u_N]^T$ is the approximate solution for exact nodal vector $\u_{ex}=u|_X=[u(x_1),\ldots,u(x_N)]^T$.
Components of $A_L$ and $A_B$ are determined by
\begin{align*}
A_L(k,j) &= \sum_{\ell\in I(y_k)}\Big(L(w_\ell\, u_j^*(\ell;\cdot))\Big)(y_k) ,\quad y_k\in Y_\Omega,\\
A_B(k,j) &= \sum_{\ell\in I(y_k)}\Big(B(w_\ell\, u_j^*(\ell;\cdot))\Big)(y_k) ,\quad y_k\in Y_\Gamma.
\end{align*}
Differential operators $L$ and $B$ should act on product $w_\ell u_j^*(\ell;\cdot)$ leading  to some complicated calculations and algorithmic complexity, as pointed out in section \ref{sect-PUmethods}.

\section{RBF-FD method}\label{sect-RBF-FD}
{In this section the RBF-FD is briefly reviewed as it is connected to the new method of section \ref{sect-RBF-FD-PU}.}
The RBF-FD arises naturally as a generalization of standard FD formulas.
For a linear differential operator $L$, the value of $(Lu)(y_k)$ can be {\em locally} approximated by values of $u$ at a
stencil $\wt X_k\subset X$ of points nearing $y_k$ by obtaining the weights
$\xi_j$ such that
\begin{equation*}
Lu(y_k) \approx \xi^T u|_{\wt X_{k}}
\end{equation*}
where the test point $y_k$ is usually assumed to be located approximately at the center of stencil $\wt X_k$.
To obtain the weight vector $\xi$, we require the stencil to reproduce all
functions spanned by RBFs $\{\phi(\cdot-x_j): x_j\in \wt X_k\}$. This happens if $\xi$ satisfies
\begin{equation*}
\sum_{x_j\in \wt X_{k}} \xi_j \phi(x_i-x_j) = L\phi(x_i-y_k), \quad x_i \in \wt X_{k},
\end{equation*}
or in matrix form
\begin{equation*}
\Phi\xi = L\phi(\cdot-y_k)|_{\wt X_k}.
\end{equation*}
It is beneficial to also augment the stencil
with polynomial terms and add matching
constraints to the associated RBF expansion coefficients.
This corresponds to requiring the
weights to further reproduce the polynomial space $\mathbb P_{n-1}(\R^d)$. The augmented linear system then becomes
\begin{equation}\label{augmented-system}
\begin{bmatrix}
\Phi & P \\
P^T & 0
\end{bmatrix}
\begin{bmatrix}
\xi\\ \nu
\end{bmatrix}
=
\begin{bmatrix}
L\phi(\cdot-y_k)|_{\wt X_k}\\
Lp(y_k)
\end{bmatrix}.
\end{equation}
{This system is uniquely solvable if and only if $\wt X_k$ is a $\mathbb P_{n-1}(\R^d)$-unisolvent set.
The $k$-th row of the differentiation matrix $A_L$ contains the RBF-FD weight vector $\xi$ of test point $y_k$ interspersed with zeros.
Zeros are corresponded to trial points outside the stencil $\wt X_k$. RBF-FD is extensively used for solving various PDE problems in engineering
and science. For more details see \cite{fornberg-flyer:2015-1} and references therein.}

\section{The new method}\label{sect-RBF-FD-PU}
{Again consider the PDE problem \eqref{Lu=f}-\eqref{Bu=g}.
In section \ref{sect-PUmethods}, the standard PU approach for solving this problem was reviewed. In this section we present an alternative
approach in which $Lu$ and $Bu$ are {\em directly} approximated by the PU approximation as}
\begin{equation}\label{direct-approx-sLsB}
Lu \approx  \sum_{\ell=1}^{N_c} w_\ell s^L_\ell=:s^L,\quad Bu \approx  \sum_{\ell=1}^{N_c} w_\ell s^B_\ell=:s^B,
\end{equation}
where $s^L_\ell$ and $s^B_\ell$ are local approximations of $Lu$ and $Bu$ in patch $\Omega_\ell$. This will be, of course, a different approach because (at least) derivatives of $w_\ell$ are not required.
Since we have a direct approximation for $Lu$ and $Bu$
{without any detour via local functions $s_\ell$ and global approximation \eqref{pu-approx},}
this approach is called the {\em direct approach}.
{Of course, if $L$ (or $B$) is the identity operator then the pure PU approximation \eqref{pu-approx} is resulted.}

Here we combine the new approach with the RBF approximation, although the above discussion is not limited to this specific approximation technique.
To this aim, we assume $s_\ell^L$ and $s_\ell^B$ in \eqref{direct-approx-sLsB} are local RBF approximations of $Lu$ and $Bu$, respectively, having representations
\begin{equation*}
s^L_\ell(x) = \sum_{j\in J_\ell} Lu_j^*(\ell;x) u(x_j), \quad s^B_\ell(x) = \sum_{j\in J_\ell} Bu_j^*(\ell;x) u(x_j), \quad x\in \Omega_\ell\cap\Omega,
\end{equation*}
where the (generalized) Lagrange function vector $Lu^*(\ell;x)$ is the solution of linear system
\begin{equation}\label{gen-lag-systm}
\begin{bmatrix}
\Phi & P\\ P^T & 0
\end{bmatrix}
\begin{bmatrix}
Lu^*(\ell;x) \\ Lv^*(\ell;x)
\end{bmatrix}=
\begin{bmatrix}
L\phi(\cdot-x))|_{X_\ell} \\ Lp(x)
\end{bmatrix}.
\end{equation}
Similarly, $Bu^*(\ell;x)$ is the solution of the same system where the operator $L$ is replaced by $B$.
Then, global approximations are written as
\begin{equation*}
s^L(x) = \sum_{\ell=1}^{N_c}\sum_{j\in J_\ell}\Big(w_\ell(x) Lu_j^*(\ell;x) \Big)u(x_j), \quad
s^B(x) = \sum_{\ell=1}^{N_c}\sum_{j\in J_\ell}\Big(w_\ell(x) Bu_j^*(\ell;x) \Big)u(x_j),
\end{equation*}
for $x\in\Omega$.
It is clear that $s^L_\ell$ and $s^B_\ell$ are identical with $Ls_\ell$ and $Bs_\ell$ on patch $\Omega_\ell$, while global approximants $s^L$ and $s^B$ are different from their counterparts $Ls$ and $Bs$.

Collocating at test points $Y_\Omega$ and $Y_\Gamma$ will yield the same system as \eqref{final-system} but with different matrix entries
\begin{align*}
A_L(k,j) &= \sum_{\ell\in I(y_k)}w_\ell(y_k) Lu_j^*(\ell;y_k) ,\quad y_k\in Y_\Omega,\\
A_B(k,j) &= \sum_{\ell\in I(y_k)}w_\ell(y_k)Bu_j^*(\ell;y_k) ,\quad y_k\in Y_\Gamma.
\end{align*}
If we set $A=\left[\begin{array}{c}A_L\\A_B\end{array}\right]$ and $\b=\left[\begin{array}{c}f|_{Y_\Omega}\\g|_{Y_\Gamma}\end{array}\right]$
then the final linear system is shortened to
\begin{equation}\label{final-system-short}
A\u=\b.
\end{equation}
Again, we note that $\u=[u_1,\ldots,u_N]^T$ is the approximate solution for the exact nodal vector $\u_{ex}=[u(x_1),\ldots,u(x_N)]$.

Comparing with the first approach, $L$ and $B$ are not needed to be operated on PU weights $w_\ell$. For example if
$L=\Delta$ then in the first method $\Delta (w_\ell u_j^*)=w_\ell\Delta u_j^* +2\nabla w_\ell\cdot\nabla u_j^* + u_j^*\Delta w_\ell$ is required
for computing the components of $A_L$, while in the second method the single term $w_\ell\Delta u_j^*$ does the whole job. So, not only all derivatives of $w_\ell$  but also many lower derivatives of Lagrange functions $u^*_j(\ell;\cdot)$ are not actually required.
At the first glance, one may expect a lost in accuracy since some terms are ignored in the new approximation.
But as we will see in the following sections, the rates of convergence
for both methods are
the same.

Consequently, the second approach suggests:
avoid approximating $u$ globally by joining local approximants $s_\ell$ and then approximating $Lu$, in favor of {\em directly} approximating $Lu$ by
joining local approximants $s^L_\ell$.
{Thus, the new method is called the {\em Direct RBF-PU} ({\em D-RBF-PU}) method.
It also has some connections to RBF-FD methods that will be explored in the following subsections.}

\subsection{Connection to RBF-FD}
There exists a tight connection to the RBF-FD method. In another point of view, the direct method of this section sets up the RBF-FD method
in a partition of unity environment.
Comparing \eqref{augmented-system} with \eqref{gen-lag-systm},
the RBF-FD weights $\xi_j$ for approximating $(Lu)(y_k)$ are (generalized) Lagrange function values $Lu_j^*(y_k)$ on set of points (stencil) $\wt X_k$. Thus, $s^L_\ell(y_k)$ is an RBF-FD approximation of $Lu(y_k)$ on stencil $\wt X_k = X_\ell$.
Since $y_k$ may belong to more than one patch $\Omega_\ell$ (precisely, $y_k\in \Omega_\ell$ for all $\ell\in I(y_k)$), all RBF-FD approximants
$s^L_\ell(y_k)$, $\ell\in I(y_k)$, are computed and then joined by
$$
\sum_{\ell\in I(y_k)}w_\ell(y_k) s^L_\ell(y_k)
$$
to form the RBF-FD partition of unity approximant $s^L(y_k)$.
Note here that $y_k$ is not approximately located at the center of patch $\Omega_\ell$ for all $\ell\in I(y_k)$.
The same argument is obviously true for the boundary operator $B$.

In RBF-FD, $M$ stencils are formed and thus $M$ local linear systems should be solved where $M$ is the number of test or evaluation points. This number is reduced to $N_c$, $N_c\ll M$, in D-RBF-PU that makes it {much} faster than the classical RBF-FD
for setting up the final linear system.
{Since RBF-FD uses a single stencil per test point, its resulting differentiation matrix is sparser.
However, as we will see, a variation of D-RBF-PU leads to a differentiation matrix that is as sparse as the one of RBF-FD.}
In theory both methods have the same order of convergence, but numerical results of section \ref{sect-numerical-results} show that the new method is more accurate in some situations
{such as for PDEs with Neumann boundary conditions.}

{At the end of the following subsection, we will show that the standard RBF-FD can be viewed as a special case of the
D-RBF-PU method.
}

\subsection{Constant{-generated} PU weight functions}\label{subsect-constantPU}
Here, we discuss two simple D-RBF-PU methods that use {piecewise constant functions $\psi_\ell$ to generate} the PU weight functions $w_\ell$ on the covering $\{\Omega_\ell\}$.
Usually, compactly supported and smooth functions (on the whole $\Omega$),
such as Wendland's functions \cite{wendland:1995-1},
are used {to generate a smooth} PU weight when derivatives are required.
As mentioned before, {no smoothness assumption on the PU weights}
is required in the new D-RBF-PU method.
A simple case is to assume
$$
\psi_\ell(x)=\chi_{\Omega_\ell}(x)=\begin{cases}1,& x\in \Omega_\ell,\\ 0,&x\notin \Omega_\ell,\end{cases}
$$
in \eqref{w-shepardform} to obtain
\begin{equation}\label{PUweight_const}
w_\ell(x) =\begin{cases} \frac{1}{|I(x)|}, & x\in \Omega_\ell,\\0,& x\notin \Omega_\ell,\end{cases}
\end{equation}
where $|I(x)|$ is the cardinality of (number of indices in) $I(x)$. In this case, local approximants $s_\ell^L(x)$ (or $s_\ell^B(x)$) have the same contribution
in global approximation $s^L(x)$ (or $s^B(x)$).

A simpler scheme will be obtained if we choose the PU weight functions as below. Assume  $\{\Omega_\ell\}_{\ell=1}^{N_c}$ is a covering for $\Omega$
and $\omega_1,\ldots, \omega_{N_c}\in \R^d$ are patches centers. For example, for circular (or spherical in 3D) patches we have $\Omega_\ell = B(\omega_\ell,\rho_\ell)$ where $\omega_\ell$ and $\rho_\ell$ are centers and radiuses, respectively.
We assume the PU weight function $w_\ell(x)$ on $\Omega_\ell$ is defined
to take the constant value $1$ if $\omega_\ell$ is the closest center to $x$ and the constant value $0$, otherwise.
For definition,
let
$$
I_{\min}(x) = \argmin_{\ell\in I(x)}\|x-\omega_\ell\|_2
$$
and $I_{\min,1}(x)$ be the first component of $I_{\min}(x)$, as $I_{\min}(x)$ may contain more than one index $\ell$.
Now, we define the weight function
\begin{equation}\label{PU-weight-const2}
w_\ell(x): = \begin{cases}
1, & \ell = I_{\min,1}(x)\\
0,& \mbox{otherwise}
\end{cases}
\end{equation}
in D-RBF-PU method. With this definition, we give the total weight $1$ to the closet patch and null weights to other patches.
In fact, a local set $X_\ell=\Omega_\ell\cap X$ is a common trial set for all
test points $y_k$ with $\|y_k-\omega_\ell\|_2\leq \|y_k-\omega_j\|_2$ for $j=1,\ldots,N_c$ and $j\neq\ell$.
In another view in a 2D domain, by drawing the {\em Voronoi tiles} of centers $\{\omega_1,\ldots,\omega_{N_c}\}$,
this means that all test points
in tile $\ell$ use the same local set $X_\ell$ as their trial set for approximation.
This results in a variation of D-RBF-PU that is much faster but as sparse as RBF-FD method.

This, also, has a connection to the {\em overlapped RBF-FD} method, recently
introduced in \cite{shankar:2017-1}.
The idea of the overlapped RBF-FD method is to use a common stencil for test points
{located on ball $B(y_k,(1-\gamma)\delta)$ where $\gamma\in[0,1]$ is the overlap parameter.
One needs to loop over all the test points, but it is required that weights computed for any point $y_i\in B(y_k,(1-\gamma)\delta)$ never be recomputed again by some
other stencil $\wt X_j$, $j\neq k$, wherein $y_i\in B(y_j,(1-\gamma)\delta)$. Thus,
the order in which the points are traversed determines the RBF-FD weights assigned to a test point.}
This idea should give a reasonable accuracy compared with the original RBF-FD but will reduce the computational costs (the costs for solving local linear systems), remarkably.
For more details, technical issues for implementation and some applications,
see \cite{shankar:2017-1,shankar-fogelsoin:2018-1,shankar-et-al:2018-1}.
{However, the D-RBF-PU with discontinuous weight \eqref{PU-weight-const2} assigns a unique closest stencil to any test point. Although the Voronoi algorithm is not used directly, the D-RBF-PU uses Voronoi tiles (constructed by patch centers) instead of balls (constructed by test points) allowing to automatically prevent any weight recomputation, and looping over the patch centers instead of the test points.}

Now, we address the question of whether
the standard RBF-FD can be viewed as a special case of the new D-RBF-PU method.
Assume that one uses the {\em heavily} overlapped covering $\{\Omega_\ell\}_{\ell=1}^M$ for $\Omega_\ell=B(\omega_\ell,\delta)$ and $\omega_\ell=y_\ell$ where $y_\ell$ for $\ell=1,\ldots,M$ are all the test points. Also, assume that the PU weight function is defined as in \eqref{PU-weight-const2}.
In this case, $\wt X_k = X_k = \Omega_k\cap X$ and $w_\ell(y_k) = \delta_{\ell k}$, $k,\ell=1,\ldots,M$. Thus,
$$
s^L(y_k) = \sum_{\ell\in I(y_k)}w_\ell(y_k) s^L_\ell(y_k)  =s_k^L(y_k) = Ls_k(y_k) = \sum_{j\in J_k}Lu_j^*(k;y_k)u(x_j) = \xi^Tu|_{\wt X_k}.
$$
Therefore, in this scenario the RBF-FD method is obtained from the D-RBF-PU method.

Finally, we note that in the classical RBF-PU method the smoothness of weight functions determines the global smoothness of the approximation.
Here, no derivatives of weight functions are needed and for the discontinuous weights (such as the above constant-{generated} PU functions) the global approximations
$s^L$ and $s^B$ are not continuous. But, this causes no drawback because $s^L$ and $s^B$ are not required to be differentiated anymore.
Of course, one can use smooth weight functions to obtain smooth approximations $s^L$
and $s^B$, if required.

\subsection{Polyharmonic kernels and scalability}
Although all RBFs in the market \cite{fasshauer:2007-1,wendland:2005-1}
can be used for approximation in local domains $\Omega_\ell\cap \Omega$, in this paper we
employ the {\em polyharmonic} kernel
\begin{equation*}
\varphi_{m,d}(r):=\left\{\begin{array}{ll}r^{2m-d}\log r,& 2m-d \mbox{ even}\\ r^{2m-d},& 2m-d \mbox{ odd}   \end{array}\right.,
\end{equation*}
for integer $m>d/2$ and assume $\phi(x)=\varphi_{m,d}(\|x\|_2)$ for $x\in\R^d$. The polyharmonic kernel
$\varphi_{m,d}$ is (up to a sign) conditionally positive definite of order $n = m-\lceil d/2\rceil+1$ and has the Beppo--Levi space
$$
\mathrm{BL}_{m}(\R^d):=\big\{f\in C(\R^d): D^\alpha f\in L_2(\R^d), \;\forall\,\alpha\in \N_0^d\,\, \mbox{with}\,\, |\alpha|=m\big\}
$$
as its {\em native space} if it is considered as a conditionally positive definite kernel of order $m$. In one and two dimensions $m=n$ while $m>n$ in higher dimensions. This means that to work with the native space $\mathrm{BL}_{m}(\R^d)$ for $d>2$, higher degree polynomials than what are actually needed to guarantee the solvability should be added to the RBF expansion.

A very useful property is that the approximation by polyharmonic kernels is {\em scalable}, allowing to avoid the instability in solving local linear systems for computing $s_\ell$, $s^L_\ell$ and $s^B_\ell$.
Here we describe this in a more detail. Assume that $X$ is a set of points in a domain $\D$ with maximum pairwise distance $h$. Assume further that the polyharmonic approximation of $D^\alpha u(x)$ for a fixed $x\in\D$ is sought.
The polyharmonic interpolation matrix $$\begin{bmatrix}\Phi & P\\ P^T & 0\end{bmatrix}$$ becomes ill-conditioned as $h$ decreases.
The conditioning of this matrix may be measured by the lower bound of $\displaystyle \lambda_{\min}(\Phi):=\min_{v: P^Tv=0}(v^T\Phi v)/v^Tv$. It is proved in \cite[chap. 12]{wendland:2005-1} that $\lambda_{\min}(\Phi)$ behaves as $h^{-4m+2d}$ independent of the number of points in $X$.
However, this problem can be overcome in a beautiful way.
If $X$ is blown up (scaled) to points $\frac{X}{h}$ of average pairwise distance $1$ and Lagrange functions $D^\alpha u_j^*$ are calculated
for the blown-up situation, then the Lagrange functions of the original situation are scaled as $h^{-|\alpha|}D^\alpha u_j^*$.
It is clear that in the blown-up situation the conditioning behaves as $\mathcal O(1)$.
For proofs and more details about the scaling property of polyharmonic kernels, see \cite{davydov-schaback:2019-1,Iske:2003-1,Iske:2013-1}.
This scaling approach works without serious instabilities for localized RBF approximations where $h$ (and the size of local domain) decreases while the number of points in $X$ is fixed; the situation in RBF-FD, RBF-PU, D-RBF-PU and all other local RBF-based methods.
In these cases, if polynomials of order $n$ are appended and if the monomials $\{x^\alpha\}_{|\alpha|< n}$ are used as a basis for $\mathbb P_{n-1}(\R^d)$, then it is better to {\em shift} the points by the center of the local domain and then scale by $h$ to benefit from the local behavior of the monomial basis functions around the origin. Note that, on behalf of the radial part, we are allowed to shift because our approximation space is shift (and rotation) invariant.

\section{Error and stability}\label{sect-error-stability}
A classical error analysis which reflects the consistency and stability bounds on the right-hand side of the final estimation is given in this section.
Since the discretization approach is {\em direct} in the sense of \cite{schaback:2013-1}, the theory given in \cite{schaback:2016-1}
for nodal meshless methods can be adapted for our analysis.
First note that since the square system
of certain meshless methods may be singular, one can apply
{\em overtesting}, i.e. choosing $M$ (the number of test points) larger than $N$ (the number
of trial points), to get a full rank system \cite{schaback:2014-1}. Thus for solvability we assume that the matrix
$A$ is set up by sufficiently thorough testing so that the matrix has rank $N\leqslant M$. The final overdetermined linear system then should be solved by a standard residual minimizer of numerical linear algebra techniques.

If we are looking for the errors in nodal values $\u_{ex}=[u(x_1),\ldots,u(x_N)]^T$, the {\em consistency} is analyzed by finding a sharp upper bound
for
$\|A\u_{ex}-\b\|_q$
where $A$ and $\b$ are the differentiation matrix and the right-hand side vector in \eqref{final-system-short}, respectively, and $\|\cdot\|_q$ is the $q$-norm in $\R^M$. According to the construction, $A({k,:})\u_{ex}=s^L(y_k)$ and $b_k = f(y_k)=Lu(y_k)$ for $y_k\in Y_\Omega$, and, similarly, $A({k,:})\u_{ex}=s^B(y_k)$ and $b_k = g(y_k)=Bu(y_k)$ for $y_k\in Y_\Gamma$.
Here, by $A(k,:)$ we mean the $k$-th row of $A$.
Thus,
for consistency we assume there exist domain and boundary error bounds
\begin{align*}
|s^L(y_k)-Lu(y_k)|\leqslant& {\varepsilon^L(y_k)}, \quad y_k\in Y_\Omega,\\
|s^B(y_k)-Bu(y_k)|\leqslant& {\varepsilon^B(y_k)}, \quad y_k\in Y_\Gamma
\end{align*}
to get
\begin{equation*}
\|A\u_{ex}-\b\|_q \leqslant \|\eep\|_q,
\end{equation*}
where $\eep$ is a vector that consists all $\varepsilon^L(y_k)$ and $\varepsilon^B(y_k)$ values.
Furthermore, assume that $\wh \u$ denotes the vector of approximate nodal values $\wh u_j$ that is obtained by some
residual minimizer numerical linear algebra algorithm that solves the system $A\u=\b$ approximately. If $\wh \u$ is calculated via minimization of the residual $\|A\u-\b\|_q$ over all $\u\in \R^N$ then
\begin{equation*}\label{linear-solver-K}
\|A\wh \u - \b\|_q =\min_{\textbf{\emph{u}}\in\R^N}\|A\u-\b\|_q\leqslant \|A\u_{ex}-\b\|_q. 
\end{equation*}
Finally, for stability we define
\begin{equation*}
C_S(A):= \sup_{\textbf{\emph{u}} \neq 0} \frac{\|\u\|_p}{\|A\u\|_q}
\end{equation*}
which is a finite constant for any $p$ and $q$ norms provided that $A$ is full rank.
Now, we can write
\begin{align*}
\|\u_{ex}-\wh \u\|_p &\leqslant C_S(A)\|A(\u_{ex}-\wh \u)\|_q\\
&\leqslant C_S(A)\big(\|A\u_{ex}-\b\|_q+\|A\wh \u-\b\|_q\big)\\
&\leqslant 2C_S(A)\|A\u_{ex}-\b\|_q\\
&\leqslant 2C_S(A)\|\eep\|_q.
\end{align*}
This is a classical error analysis where the right-hand side contains the product of a stability constant and a consistency bound.
The remaining parts of this section concern the estimations of these ingredients.
\subsection{Consistency}
The following theorem states that the partition of unity approximant is at least as good
as its worst local approximant.
\begin{theorem}\label{thm-PUdirect-err}
Let $\Omega\subset \R^d$ be bounded and $\{\Omega_\ell\}_{\ell=1}^{N_c}$ be an open and bounded
covering of $\Omega$ with a partition of unity $\{w_\ell\}_{\ell=1}^{N_c}$.
If $u$ allows the action of $L$ 
and in each region $\Omega_\ell\cap\Omega$, $Lu$ is approximated
by a function $s^L_\ell$ such that
$$
\|Lu-s^L_\ell\|_{L_\infty(\Omega_\ell\cap\Omega)} \leqslant \varepsilon_\ell^L,
$$
then the global approximant $s^L$ satisfies
\begin{equation}\label{error-bound-Lu-sL}
|Lu(x)-s^L(x)|\leqslant \max_{\ell\in I(x)}\varepsilon_\ell^L{=:\varepsilon^L(x)}, \quad x\in \Omega
\end{equation}
\end{theorem}
\begin{proof}
Since $\{w_\ell\}$ forms a partition of unity, we simply have for any $x\in\Omega$
\begin{align*}
|Lu(x)-s^L(x)| &\leqslant \sum_{\ell\in I(x)} w_\ell(x) \big|Lu(x)-s^L_\ell(x)\big| \\
&\leqslant  \sum_{\ell\in I(x)} w_\ell(x) \big\|Lu-s^L_\ell\big\|_{L_\infty(\Omega_\ell\cap\Omega)}\\
&\leqslant  \sum_{\ell\in I(x)} w_\ell(x) \varepsilon_\ell^L\\
&\leqslant \max_{\ell\in I(x)}\varepsilon_\ell^L,
\end{align*}
which completes the proof.
\end{proof}

Similarly, if $u$ allows the action of $B$ 
and
$$
\|Bu-s^B_\ell\|_{L_\infty(\Omega_\ell\cap\Gamma)} \leqslant \varepsilon_{\ell}^B\,,
$$
then
the estimation
$$
|Bu(x)-s^B(x)|\leqslant \max_{\ell\in I(x)}\varepsilon_{\ell}^B{=:\varepsilon^B(x)}\,,\quad x\in\Gamma
$$
holds true.

Recall that if the PU functions $w_\ell$ are continuous on the whole $\Omega$, then the global approximation $s^L$ is also continuous. It is not the case for the constant-generated PU weights \eqref{PUweight_const} and \eqref{PU-weight-const2}. {However, the only property that is used in the proof of Theorem \ref{thm-PUdirect-err} is the partition of unity property.}
Thus, the analysis below holds also true for the discontinuous PU weights and in particular for the standard RBF-FD method.

{For weight \eqref{PUweight_const} at each test point $x$ where the number of overlapping patches change, and for weight \eqref{PU-weight-const2} at each point $x$ on the edge of the Voronoi regions of the patch centers,}
we have a discontinuity in $w_\ell(x)$ that imposes a discontinuity in global approximation $s^L(x)$ by a jump value of order
$$\max_{\ell\in I(x)}\varepsilon_{\ell}^L$$
provided that $Lu$ is continuous.

Despite the analysis of the standard derivatives of PU approximation (see \cite[Section 15.4]{wendland:2005-1}), here PU functions $w_\ell$ need no smoothness and controlling assumption such as that given in
\eqref{Dw_ell-bound}.
Besides,
$|I(x)|=\sum \chi_{\Omega_\ell}(x)$ is not assumed to be uniformly bounded in $\Omega$, although
this assumption will increase the computational efficiency of both algorithms.
See \cite[Definition 15.18]{wendland:2005-1} for comparison.

{We will need to work with a variety of Sobolev spaces.
For open set $\Omega\subset\R^d$, $k\in \mathbb N_0$ and $1 \leqslant p<\infty$ the Sobolev spaces $W_p^k(\Omega)$
consist of all
$u$ with weak derivatives $D^\alpha u\in L_p(\Omega)$, $|\alpha|\leqslant k$. The (semi-)norms
$$
|u|_{W^k_p(\Omega)}^p:=\sum_{|\alpha|=k}\|D^\alpha u\|_{L_p(\Omega)}^p  \;\; \mbox{and}\;\; \|u\|_{W^k_p(\Omega)}^p:=\sum_{|\alpha|\leqslant k}\|D^\alpha u\|_{L_p(\Omega)}^p,
$$
are associated with these spaces.
The case $p=\infty$ is defined by
$$
|u|_{W^k_p(\Omega)}:=\sup_{|\alpha|=k}\|D^\alpha u\|_{L_\infty(\Omega)} \;\; \mbox{and}\;\; \|u\|_{W^k_p(\Omega)}:=\sup_{|\alpha|\leqslant k}\|D^\alpha u\|_{L_\infty(\Omega)}.
$$
The Hilbert space $W^k_2(\Omega)$ is usually denoted by $H^k(\Omega)$.
To introduce the Sobolev spaces on the boundary,
we assume that the bounded domain $\Omega$ has a $C^k$ boundary $\partial \Omega$ and $\partial \Omega\subset \sum_{j=1}^K U_j$, where $U_j\subset\R^d$ are open sets. Moreover, we assume that $U_j$ are images of $C^k$-diffeomorphisms $\varphi_j:B\to U_j$  where $B=B(0,1)$ denotes here the unit ball in $\R^{d-1}$. Finally, if we assume that $\{w_j\}$ is a
partition of unity with respect to $\{U_j\}$ then the Sobolev norms on $\partial \Omega$ can be defined via
$$
\|u\|_{W^k_p(\partial\Omega)}^p:=\sum_{j=1}^K \|(uw_j)\circ\varphi_j\|_{W^k_p(B)}^p, \quad \|u\|_{W^k_\infty(\partial\Omega)}:=\sup_{1\leqslant j\leqslant K} \|(uw_j)\circ\varphi_j\|_{W^k_\infty(B)}.
$$
These norms are independent of the chosen atlas $\{U_j,\varphi_j\}$.
}

From here on, we assume that $L$ is a linear differential operator of the form
\begin{equation*}
Lu(x) = \sum_{|\alpha|\leqslant k_L}a_\alpha(x) D^\alpha u(x),\quad x\in\Omega,
\end{equation*}
where $k_L$ is the order $L$ and the coefficients $a_\alpha$ lie in space $L_\infty(\Omega)$.
We can also simply show that
\begin{equation}\label{Lu-bound-u}
\|Lu\|_{L_\infty(\Omega)}\leqslant C_a\|u\|_{W^{k_L}_\infty(\Omega)},
\end{equation}
{where $C_a = \max_{|\alpha|\leqslant k_L}\|a_\alpha\|_{L_\infty(\Omega)}$}.
Since different types of boundary conditions may be imposed on $\Gamma$,
we assume $\Gamma=\Gamma_1\cup\cdots\cup \Gamma_T$ such that $\Gamma_i\cap\Gj=\emptyset$ for $i\neq j$,
and
\begin{equation*}
Bu(x) =\sum_{j=1}^T\chi_{\Gj}(x) \sum_{k=0}^{k_j}b_{j,k}(x) \frac{\partial^k u}{\partial \nu^k}(x), \quad x\in \Gamma,
\end{equation*}
where $\nu$ is the outward normal to the boundary, and $\Gj$ are of smoothness class $C^{k_j}$ for $1\leqslant j\leqslant T$ in which
$k_j$ is the order of $B$ on $\Gj$. Also we assume $b_{j,k}\in L_\infty(\Gj)$. Here $\chi_{\Gj}(x)$ is the characteristic function of set $\Gj$, i.e., it is $1$ if $x\in\Gj$ and $0$ otherwise.
It is not difficult to prove
that
\begin{equation}\label{Bu-bound-u}
\|Bu\|_{L_\infty(\Gj)}\leqslant C_{b_j}\|u\|_{W^{k_j}_\infty(\Gj)},\; 1\leqslant j\leqslant T.
\end{equation}
{where $C_{b_j} = \max_{1\leqslant k\leqslant k_j}\|b_{j,k}\|_{L_\infty(\Gj)}$.}
If we set $k_B=\max\{k_j:1\leqslant j\leqslant T\}$ and assume $\Gamma$ is of smoothness class $C^{k_B}$ then the norm on the right-hand side of
\eqref{Bu-bound-u} can be overestimated by $\|u\|_{W^{k_B}_\infty(\Gj)}$.

In order to achieve high order convergence, the regularity of true solution $u$ needs
to be higher than what is strictly required by the problem itself.
In the following we assume
{that the domain, the boundary, the boundary conditions and the right-hand side function $f$ allow the unique solution}
$u\in H^m(\Omega)\cap W_\infty^{k_B}(\Omega)$
for some $m> k_L+d/2$.

The local bounds $\varepsilon_\ell^L$ and $\varepsilon_{\ell}^B$ for polyharmonic kernels in the Sobolev norms can be estimated by applying the following theorem.
\begin{theorem}\cite[Theorem 11.36]{wendland:2005-1}\label{thm-polyharmonic-err1}
Let $m>k+d/2$. Suppose that $\D\subset \R^d$ is open and bounded and satisfies
an interior cone condition. Consider the polyharmonic kernel $\varphi_{m,d}$
as conditionally positive
definite of order $m$. Then the error between $u\in H^m(\D)$ and its polyharmonic interpolant $s$ on $X\subset\D$ can be
bounded by
$$
|u-s|_{W_p^k(\D)}\leqslant C h_{X,\D}^{m-k-d(1/2-1/p)_{+}}|u|_{H^{m}(\D)}
$$
for $1\leqslant p\leqslant\infty$ and sufficiently small $h_{X,\D}$. If we use norms instead of seminorms we have
$$
\|u-s\|_{W_p^k(\D)}\leqslant C h_{X,\D}^{m-k-d(1/2-1/p)_{+}}\|u\|_{H^{m}(\D)}.
$$
Here, by $x_{+}$ we mean $\max\{x,0\}$.
\end{theorem}

The case $p=\infty$ reduces to error bound
$$
|u-s|_{W_\infty^k(\D)}\leqslant C h_{X,\D}^{m-k-d/2}|u|_{H^{m}(\D)},
$$
and will be used in the sequel.
If polynomials of higher orders are appended or the target function has an arbitrary smoothness, the analysis given in
\cite{davydov-schaback:2019-1} for `optimal stencils in Sobolev spaces' can be adapted, instead.
According to \cite{davydov-schaback:2019-1}, for
$u\in H^m(\D)$, $m>k+d/2$, the achievable rate for a large enough stencil $X\subset \D$ turns out to be
\begin{equation}\label{optimal_rate_stencils}
|u-s|_{W_\infty^k(\D)}\leqslant C\begin{cases}h^{m-k-d/2}|u|_{H^m(\D)}, & m<q+d/2,\\
h^{q-k}|u|_{H^m(\D)}, & m>q+d/2,\\
h^{m-k-d/2-\epsilon}|u|_{H^m(\D)}, & m=q+d/2,\, \epsilon \mbox{ arbitrary small},
\end{cases}
\end{equation}
where $h$ is the stencil size and $q$ is the maximal order of polynomials on which the approximation is exact.
This means that one cannot have a convergence rate better than $m - k -d/2$ for functions in $H^m(\D)$, no matter how many nodes are used for approximation, where they are placed and how large $q$ is chosen. On the other side, for a fixed node set $X$ the convergence
rate in any Sobolev space $H^m(\D)$ cannot be better than $q - k$ no matter how large $m$ is.
{This optimal convergence rate in $H^m(\D)$ can be obtained via polyharmonic kernels $\varphi_{m,d}$ provided that the underlying set allows exactness on
polynomials of order $q = \lfloor m - d/2\rfloor  + 1$. If the target function lies in some $H^n(\D)$ with $n>m$, and if we manage
to have a scalable stencil of polynomial exactness $q>n-d/2$,
no matter how we get it, maybe from
polyharmonics $\varphi_{m,d}$ with a smaller $m$ but additional polynomials up to order $q$, then according to \cite{davydov-schaback:2019-1} we have a stencil with optimal order in $H^n(\D)$. If $q<n-d/2$ the order remains at $q-k$ no matter how large $n$ is.
Thus the error bound \eqref{optimal_rate_stencils} by replacing $m$ by $n$ is still valid for $u\in H^n(\D)$. This means that in a PHS approximation if the target function is smooth enough then the order of convergence is fully determined by the amount of appended polynomials, no matter how large or small the exponent of the radial part is. Of course, the radial part determines the minimal order of polynomials that should be augmented to the expansion to obtain a unique stencil.
}

To treat the approximation at the boundary points, we need a kind of {\em trace theorem} holding for infinity norms.
If $\D$ is a bounded and open set in $\R^d$ with a $C^1$ boundary then
the {\em Morrey's inequality} implies that there exists $C_M>0$ such that for all $u\in W^1_\infty(\D)$
$$
\|u\|_{C^{0,1}(\D)}\leqslant C_M \|u\|_{W^1_\infty(\D)},
$$
where $C^{0,1}(\D)$ is the space of Lipschitz functions on $\D$. Since $u$ is bounded and Lipschitz, we can extend its domain to $\overline \D$ by continuity. Hence we have a trace operator $W^1_\infty(\D)\to L_\infty(\partial \D)$ with
$$
\|u|_{\partial \D}\|_{L_\infty(\partial\D)}\leqslant \|u\|_{L_\infty(\D)}\leqslant \|u\|_{W^1_\infty(\D)}.
$$
Considering the first inequality, this also shows that if $\D$ has a $C^{k+1}$ boundary and $u\in W^{k+1}_\infty(\D)$ then
\begin{equation*}
\|u\|_{W^k_\infty(\partial\D)}\leqslant C \|u\|_{W^k_\infty(\D)}.
\end{equation*}

Now, using Theorems \ref{thm-PUdirect-err} and \ref{thm-polyharmonic-err1} and the above discursion, we have the following error estimation.
\begin{theorem}\label{thm-LusL-error}
Let $\Omega\subset \R^d$ be an open and bounded domain with a $C^{k_B+1}$ boundary. Let $\{\Omega_\ell\}_{\ell=1}^{N_c}$ be an open and bounded
covering of $\Omega$ with PU functions $w_\ell$.
Suppose that all sets $\Omega_\ell\cap\Omega$ satisfy
interior cone conditions. Obtain all local approximants $s^L_\ell$ and $s^B_\ell$ using the polyharmonic kernel $\varphi_{m,d}$
where it is considered as a conditionally positive definite of order $m$.
If $s^L$ and $s^B$ are the direct PU approximations of $Lu$ and $Bu$, respectively, then
\begin{align*}
|Lu(x)-s^L(x)|&\leqslant C h_{X,\Omega}^{m-k_L-d/2}\|u\|_{H^m(\Omega)}, \quad x\in \Omega,\\
|Bu(x)-s^B(x)|&\leqslant C h_{X,\Omega}^{m-k_B-d/2}\|u\|_{H^m(\Omega)}, \quad x\in \Gamma,
\end{align*}
hold for sufficiently small fill distance $h_{X,\Omega}$ and all $u\in H^m(\Omega)\cap W_\infty^{k_B}(\Omega)$ with $m>k_L+d/2$.
{If, in additions, polynomials of higher order $q$ are appended to the PHS kernel $\varphi_{m,d}$ and
$u\in H^{n}(\Omega)\cap H^{k_B}(\Omega)$, $n>k_L+d/2$, then
\begin{align*}
|Lu(x)-s^L(x)|&\leqslant C
\begin{cases}
h_{X,\Omega}^{n-k_L-d/2}\|u\|_{H^{n}(\Omega)}, & n<q+d/2,\\
h_{X,\Omega}^{q-k_L}\|u\|_{H^{n}(\Omega)}, & n>q+d/2,\\
h_{X,\Omega}^{n-k_L-d/2-\epsilon}\|u\|_{H^{n}(\Omega)}, & n=q+d/2,\, \epsilon \mbox{ arbitrary small},
\end{cases}
\end{align*}
provided that the stencil sizes are proportional to the fill distance $h_{X,\Omega}$. The same bound holds true for $|Bu(x)-s^B(x)|$ by replacing $k_L$ by $k_B$.
}
\end{theorem}
\begin{proof}
To prove the first error bound, according to Theorem \ref{thm-PUdirect-err}, it is sufficient to estimate the upper bounds $\varepsilon_\ell^L$.
Using the fact that $s^L_\ell = Ls_\ell$ on $\Omega_\ell\cap\Omega$, we can write for any $\ell\in\{1,\ldots,N_c\}$,
\begin{align*}
\|Lu-s^L_\ell\|_{L_\infty(\Omega_\ell\cap\Omega)}&=\|Lu-Ls_\ell\|_{L_\infty(\Omega_\ell\cap\Omega)}\\
&\leqslant C_a \|u-s_\ell\|_{W^{k_L}_{\infty}(\Omega_\ell\cap\Omega)}\\
&\leqslant C_a Ch_{X_\ell,\Omega_\ell\cap\Omega}^{m-k_L-d/2}\|u\|_{H^m(\Omega_\ell\cap\Omega)}\\
&=:\varepsilon_\ell^L,
\end{align*}
where \eqref{Lu-bound-u} and Theorem \ref{thm-polyharmonic-err1} are applied in the second and third lines, respectively.
Then $h_{X_\ell,\Omega_\ell\cap\Omega}\leqslant h_{X,\Omega}$ and $\|u\|_{H^m(\Omega_\ell\cap\Omega)}\leqslant \|u\|_{H^m(\Omega)}$
finish the proof.
For the  error estimation on the boundary, we first modify the local domains $\Omega_\ell\cap\Omega$ to some open domains $\wt \Omega_\ell$
such that $\Omega_\ell\cap\Omega\subseteq\wt \Omega_\ell\subset\Omega$,
$h_{X_\ell,\wt\Omega_\ell}=C h_{X_\ell,\Omega_\ell\cap\Omega}$ and $\wt\Omega_\ell$ have $C^{k_B+1}$ boundaries. Then
\begin{align*}
\|Bu-s^B_\ell\|_{L_\infty(\Omega_\ell\cap\Gj)}&=\|B u-Bs_\ell\|_{L_\infty(\Omega_\ell\cap\Gj)}\\
&\leqslant C_{b} \|u-s_\ell\|_{W^{k_B}_{\infty}(\Omega_\ell\cap\Gj)}\\
&\leqslant C_{b} \|u-s_\ell\|_{W^{k_B}_{\infty}(\partial\wt\Omega_\ell)}\\
&\leqslant C_{b} \|u-s_\ell\|_{W^{k_B}_{\infty}(\wt\Omega_\ell)}\\
&\leqslant C_{b} Ch_{X_\ell,\wt\Omega_\ell}^{m-k_B-d/2}\|u\|_{H^m(\wt\Omega_\ell)}\\
&=:\varepsilon_{\ell}^B,
\end{align*}
where $C_b$ is the maximum of $C_{b_j}$'s. Finally, $h_{X_\ell,\wt\Omega_\ell}=Ch_{X_\ell,\Omega_\ell\cap\Omega}\leqslant Ch_{X,\Omega}$ and $\|u\|_{H^m(\wt\Omega_\ell)}\leqslant \|u\|_{H^m(\Omega)}$
complete the proof.
We note here that, not only the boundary points on $\Omega_\ell\cap\Gj$ but also
the interior points in $\Omega_\ell\cap\Omega$ are all contributed in the approximation process of $Bu(x)$ for $x\in\Omega_\ell\cap\Gj$.
Thus, the final bound should contain the domain fill distance $h_{X,\Omega}$, as it does.

{
The proof of the second bound is the same by using the discussion right after Theorem \ref{thm-polyharmonic-err1}.
}
\end{proof}

Theorem \ref{thm-LusL-error} proves that $\varepsilon^L(y_k) = \mathcal O(h_{X,\Omega}^{m-k_L-d/2})$ for $\{k: y_k\in Y_\Omega\}$ and
$\varepsilon^B(y_k) = \mathcal O(h_{X,\Omega}^{m-k_B-d/2})$ for $\{k: y_k\in Y_{\Gamma}\}$ provided that $u\in H^m(\Omega)\cap W_\infty^{k_B}(\Omega)$.
In general, $\|\eep\|_\infty=\mathcal O(h_{X,\Omega}^{m-k_L-d/2})$ as $k_B<k_L$ and $h_{X,\Omega}$ is assumed to be sufficiently small.
If $u$ is smooth enough, $L$ and $B$ are scalable and polynomials of higher order $q$ are appended to the PHS expansion then the rates will be improved to
$h_{X,\Omega}^{q-k}$ for $k=k_L, k_B$, and the consistency order $\|\eep\|_\infty=\mathcal O(h_{X,\Omega}^{q-k_L})$ will be resulted.

\subsection{Stability}
Despite the lack of a theoretical bound even for simple operators $L=\Delta$ and $B=Id$, Schaback \cite{schaback:2016-1} has proposed some
numerical estimators for the stability constant $C_S(A)$ for an arbitrary matrix $A$.
For example, in case $p=q=2$,
$$
C_S(A)=\left(\min_{1\leqslant j\leqslant N}\sigma_j\right)^{-1}
$$
for the $N$ positive {singular values} $\sigma_1,\ldots,\sigma_N$ of $A$, and these are obtainable by
singular value decomposition (SVD).
Also, the $(q,p)$-norm of the pseudoinverse of $A$, defined by
$$
\|A^\dag\|_{q,p}:=\sup_{\emph{\textbf{u}} \,\neq 0}\frac{\|A^\dag \u\|_q}{\| \u\|_p},
$$
overestimates $C_S(A)$. Finally,
a simple possibility, restricted to square systems, is to use the fact that \textsc{Matlab}'s
$\texttt{condest}$ command estimates the $L_1$ condition number, which is the $L_\infty$ condition number
of the transpose. Thus
\begin{equation*}
\wt C_S(A):=\frac{\texttt{condest}(A^T)}{\|A\|_\infty}
\end{equation*}
is an estimate of the $L_\infty$ norm of $A^{-1}$. This is computationally very cheap for sparse
matrices, however an extension to
non-square matrices is missing.

Although numerical results of section  \ref{sect-numerical-results} show an excellent stability for special $L$ and $B$ operators,
it is left for a future work to theoretically estimate $C_S(A)$ in terms of discretization parameters and behaviour of operators.
This is an open problem not only for the method of this paper but also for all previous unsymmetric local meshless (RBF-based or else) methods.
{See \cite{davydov:2020-1,tominec-et-al:2020-1} for recent attempts to tackle a similar problem in least squares settings for the RBF-FD method.}

\section{{Numerical results}}\label{sect-numerical-results}

In this section, some numerical results of the D-RBF-PU method and comparisons with other local RBF-based methods are given.
We consider the Poisson equation and the standard diffusion (heat) equation with Dirichlet and Neumann boundary conditions (BC) in two and three dimensions.

All algorithms are implemented in \textsc{Matlab} and executed on a machine with an Intel Core i7 processor, 4.00 GHz and
16 GB RAM.

{\em Domains:}
The box domain
$\Omega_B:= (0,1)^2,$
the circular domain
$
\Omega_C:=\{x\in \R^2: \|x\|_2<1\},
$
and the non-convex domain with smooth boundary, defined using polar coordinates as \cite{larsson-et-al:2017-1}
$$
\Omega_S:= \{x=(r,\theta): r< 0.7 + 0.12(\sin 6\theta +\sin 3\theta)=:r_S,\; \theta\in[0,2\pi)\},
$$
are used for experiments in $\R^2$. In $\R^3$ we consider the unit ball
$
\Omega_U:=\{x\in\R^3: \|x\|_2<1\},
$
and the non-convex domain
$$
\Omega_Q: = \{x = (r,\theta,\varphi): r<r_Q(\theta,\varphi),\; \theta\in [0,2\pi), \, \varphi\in [0,\pi]\},
$$
where $r_Q = \big[1 + \sin^2(2 \sin \varphi \cos \theta) \sin^2(2 \sin \varphi \sin \theta) \sin^2(2 \cos \varphi)\big]^{1/2}$.
The 3D domains $\Omega_U$ and $\Omega_Q$ are shown in Figure \ref{fig_domain1}, and the 2D domains $\Omega_C$ and
$\Omega_S$ are shown in Figure \ref{fig_covering1}.
\begin{figure}[!h]
\begin{center}
\includegraphics[width=5cm]{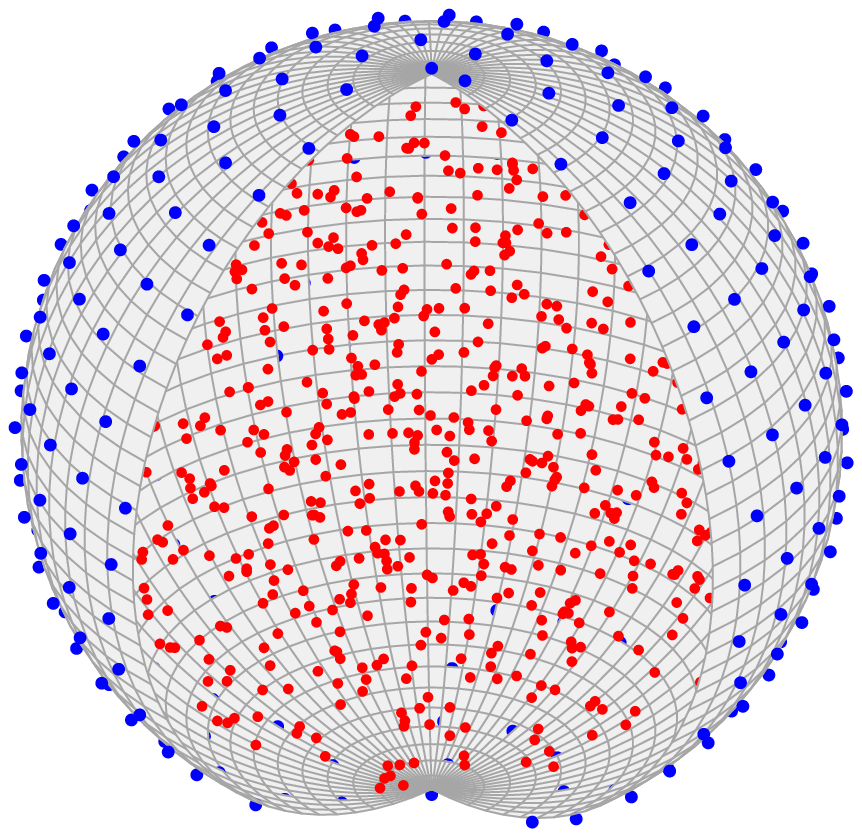}\includegraphics[width=5cm]{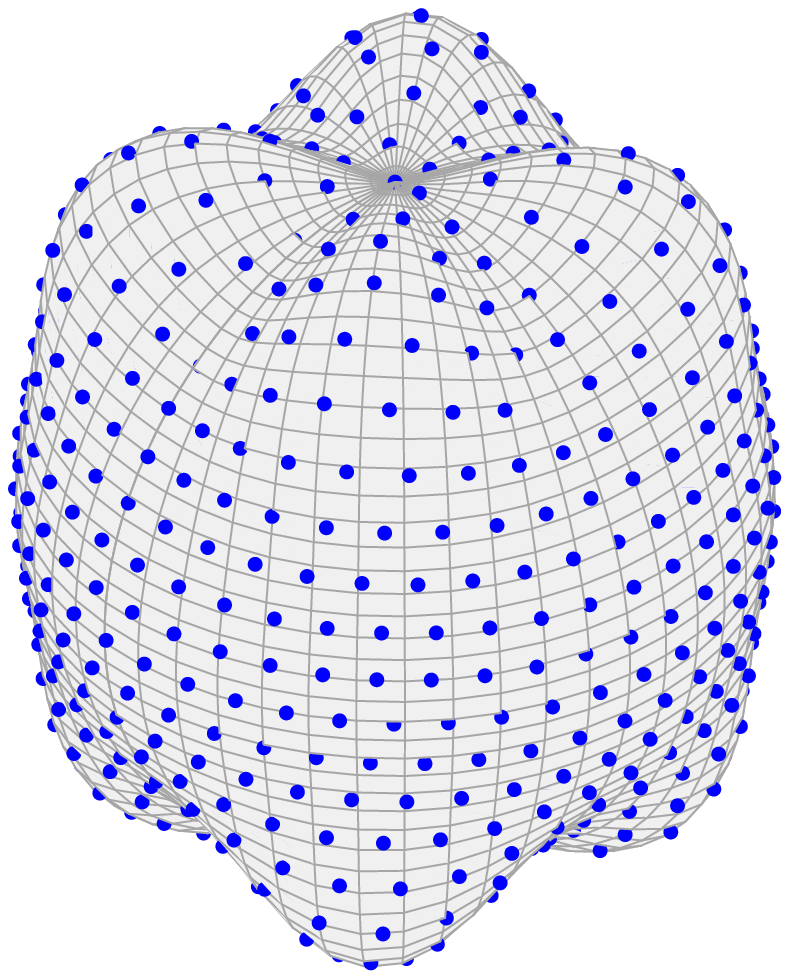}
\caption{\small{Three dimensional domains $\Omega_U$ (left) and $\Omega_Q$ (right), together with boundary points. In the case of $\Omega_U$, some of the internal trial points are shown in red.}}\label{fig_domain1}
\end{center}
\end{figure}
\begin{figure}[!h]
\begin{center}
\includegraphics[width=12cm]{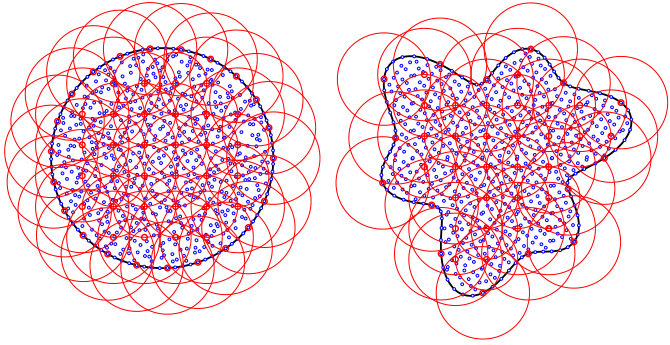}
\caption{\small{690 Halton points on domain $\Omega_C$ (left) and 681 Hammersley points on domain $\Omega_S$ (right), together with circular patches.}}\label{fig_covering1}
\end{center}
\end{figure}

{\em Boundary conditions:}
The Neumann boundary condition is imposed on  top and bottom sides of $\partial\Omega_B$, on upper semi-circle of $\partial\Omega_C$, i.e., on $\{x=(1,\theta)\in \partial\Omega_C: \theta\in[0,\pi]\}$, on upper curve of $\partial\Omega_S$, i.e., on $\{x=(r_S,\theta)\in \partial\Omega_S: \theta\in[0,\pi]\}$, on north surface of $\partial\Omega_U$, i.e., on $\{x=(1,\theta,\varphi)\in \partial\Omega_U: \theta\in [0,2\pi),\, \varphi\in(\pi/2,\pi]\}$, and on north surface of $\partial\Omega_Q$, i.e., on $\{x=(r_Q,\theta,\varphi)\in \partial\Omega_Q: \theta\in [0,2\pi),\,\varphi\in(\pi/2,\pi]\}$. Other parts of boundaries are constrained by Dirichlet boundary conditions.
For comparison in some experiments, we may also use a pure Dirichlet or a pure Neumann boundary condition.

{\em Sets of points:}
Scattered trial and test points with fill distance $h=h_{X,\Omega}$ are used in experiments.
Halton points on $\Omega_B$, $\Omega_C$ and $\Omega_U$, Hammersley points on $\Omega_S$ and gridded points on $\Omega_Q$ are employed.
We construct the points on a cube and use their restriction to the domain $\Omega$.
The fill distance is approximated by $h \approx (c_d N)^{-1/d}$ where $N$ is the number of trial points, $d$ is the dimension, and $c_d=\mathrm{volume}(B)/\mathrm{volume}(\Omega)$ where $B$ here is a smallest possible cube that contains $\Omega$.
Boundary points $\{(1,\theta_k):1\leqslant k\leqslant |Y_\Gamma|\}$ and $\{(r_S,\theta_k):1\leqslant k\leqslant |Y_\Gamma|\}$ are used on
$\partial \Omega_C$ and $\partial \Omega_S$, respectively, where $\{\theta_k\}$ is a set of equidistance points on $[0,2\pi)$.
The boundary points on  $\partial\Omega_U$ are constructed by the {\em equal area partitioning} algorithm \cite{saff-kuijlaars:1997-1}.
On $\partial\Omega_Q$ we use the projected points from $\partial\Omega_U$. See Figure \ref{fig_domain1}.
In all cases, the fill distance on the boundary (the number of points on the boundary) is adjusted to the fill distance of internal points to be approximately of the same size.

The PU covering $\{\Omega_\ell\}=\{B(\omega_\ell,\rho_\ell)\cap\Omega\}$ with centers $\{\omega_1,\omega_2,\ldots,\omega_{N_c}\}\subset\overline \Omega$ is used. Gridded patch centers with {horizontal and vertical} distances $h_{c}$ are used inside $\Omega$, and boundary centers with a distance of order $h_c$ are constructed with the same techniques discussed above for the boundary test points. For points far from the boundary, the radiuses of patches are assumed to be constant (independent of $\ell$) and proportional to $h_{c}$, i.e.,
$\rho_\ell=\rho = C_c h_c$ where
$C_c$ determines the amount of overlap among patches. But, for points on and adjacent to the boundary (up to a radial $h_c$-distance from the boundary) we increase the radius $\rho$ by a factor of $1.5$. See also subsection \ref{sect-overlap-const} below.
In Figure \ref{fig_covering1}, a set of trial points and a covering are shown on domains $\Omega_C$ and $\Omega_S$.

{\em Weight functions:}
As a smooth PU weight, the function
\begin{equation}\label{wend_weight}
\psi_\ell=\psi(\|\cdot-\omega_\ell\|_2/\rho_\ell),\quad \psi(r)=(1-r)_{+}^6(35r^2+18r+3),
\end{equation}
is used in \eqref{w-shepardform} where $\psi(r)$ is the $C^4$ compactly supported Wendland's function \cite[Chap. 9]{wendland:2005-1}.
In some experiments, constant-generated PU functions \eqref{PUweight_const} and \eqref{PU-weight-const2} are also employed.
We will see that in some cases, a combination of the smooth weight (for boundary patches) and the constant-generated weight \eqref{PU-weight-const2}
(for internal patches) increases the efficiency of the method in terms of accuracy, complexity and sparsity.

{\em Kernels:}
Polyharmonic splines $\varphi(r)=r^6\log(r)$ (PHS6) and $\varphi(r)=r^8\log r$ (PHS8) are used in 2D, while $\varphi(r)=r^5$ (PHS5) and $\varphi(r)=r^7$ (PHS7) are applied in 3D cases. These RBFs are conditionally positive definite of orders $n=4$, $5$, $3$ and $4$, respectively. Thus, polynomial spaces $\mathbb P_{n-1}(\R^d)$ are augmented to guarantee the solvability of approximation problems.
However, we will also use polynomials of higher orders to observe the effect of polynomials on the convergence rates.
In the legend of figures, (for example) by PHS5+P2 we mean an approximation through PHS kernel $r^5$ augmented with polynomials of degree at most $2$ (order $3$).

{\em True solutions:}
In 2D, the known Franke's function \cite{franke:1982-1}, and in 3D the function
$$
u(x) = \sin\left(\frac{\pi(x^1-0.5)x^3}{\log(x^2+3)}\right), \quad x=(x^1,x^2,x^3)\in \R^3
$$
are assumed to be the true solutions for the steady state problems \cite{larsson-et-al:2017-1}.
The right-hand side function $f$ and boundary conditions are obtained, accordingly.

For the time-dependent problem in 2D experiments, the prescribed true solution
$$
u(x, t)= 1 + \sin(\pi x^1) \cos(\pi x^2)\exp(-\pi t)
$$
is used \cite{shankar:2017-1}.
The forcing term that makes this solution hold is given by
$f (x, t) = \pi (2\pi\kappa - 1) \sin(\pi x^1) \cos(\pi x^2)\exp(-\pi t)$.
In 3D, we use the true solution
$$
u(x, t)= 1 + \sin(\pi x^1) \cos(\pi x^2)\sin(\pi x^3)\exp(-\pi t)
$$
with
$f (x, t) = \pi (3\pi\kappa - 1) \sin(\pi x^1) \cos(\pi x^2)\sin(\pi x^3)\exp(-\pi t)$.
Boundary conditions are obtained by the restriction of exact solutions and/or their
derivatives on the boundary.

{\em Overtesting:}
Overtesting is not applied at all, because the results show that square systems for both regular and irregular points
are full rank and extremely stable.

{\em Convergence plots:} Since $h=\mathcal O(N^{-1/d})$, we plot the errors and the stability numbers vs. $N^{1/d}$. All convergence plots are on a log-log scale. Numerical convergence orders are obtained by the linear least squares fitting to error values, and are written alongside the figure legends.

\subsection{Overlap constant $C_c$}\label{sect-overlap-const}
As pointed out above, in this study we use an overlapping covering that consists of balls $B(\omega_\ell,\rho_\ell)$ for $\ell=1,\ldots,N_c$.
We assume the set $\{\omega_1,\ldots,\omega_{N_c}\}$ of covering centers has vertical distance $h_c$ which is proportional to fill distance $h$
of trial points $X$. We use $h_c=4h$ in all experiments. The covering radius
$$\rho_\ell=\rho=C_ch_c$$
affects both the accuracy of numerical solution and the sparsity of final linear system. %
For points $\omega_\ell$ on and close to the boundary $\partial \Omega$, the number of trial points $X\cap B(\omega_\ell,\rho)$ is decreased by more than a $\frac{1}{2}$ factor. Thus, we should increase the radius of patches by a factor of more than $2^{1/d}$ ($\approx 1.41$ in 2D and $\approx 1.26$ in 3D).
To be sure that we have enough local trial points, we increase $\rho$ by a factor of $1.5$ for such patches.

\begin{figure}[!h]
\begin{center}
\includegraphics[width=5cm]{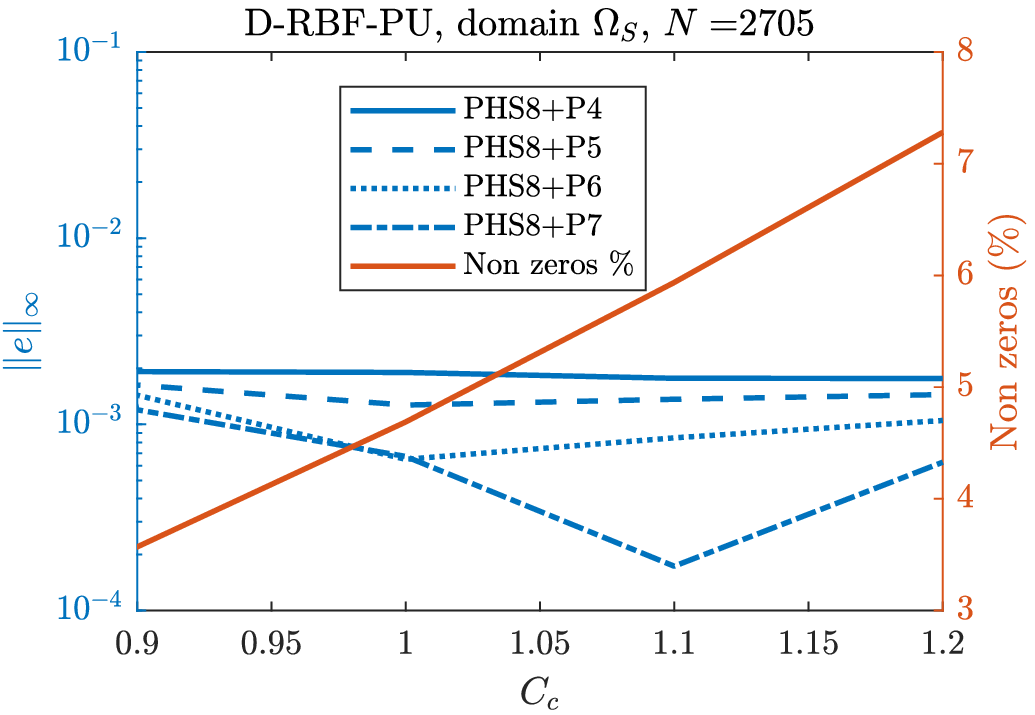}\includegraphics[width=5cm]{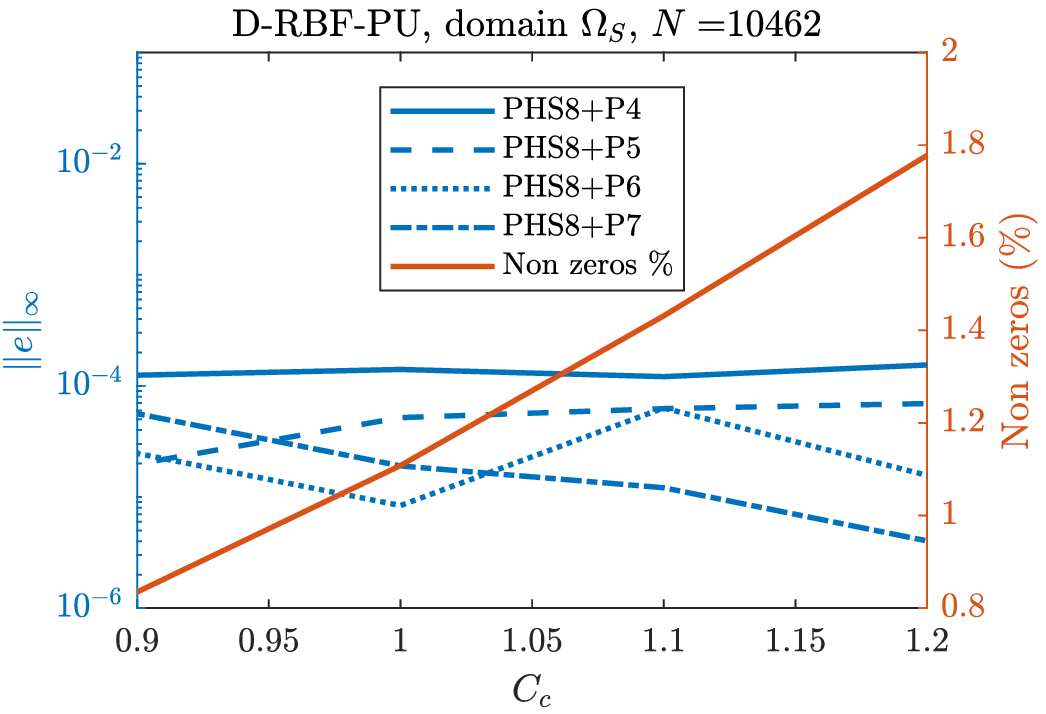}\\
\includegraphics[width=5cm]{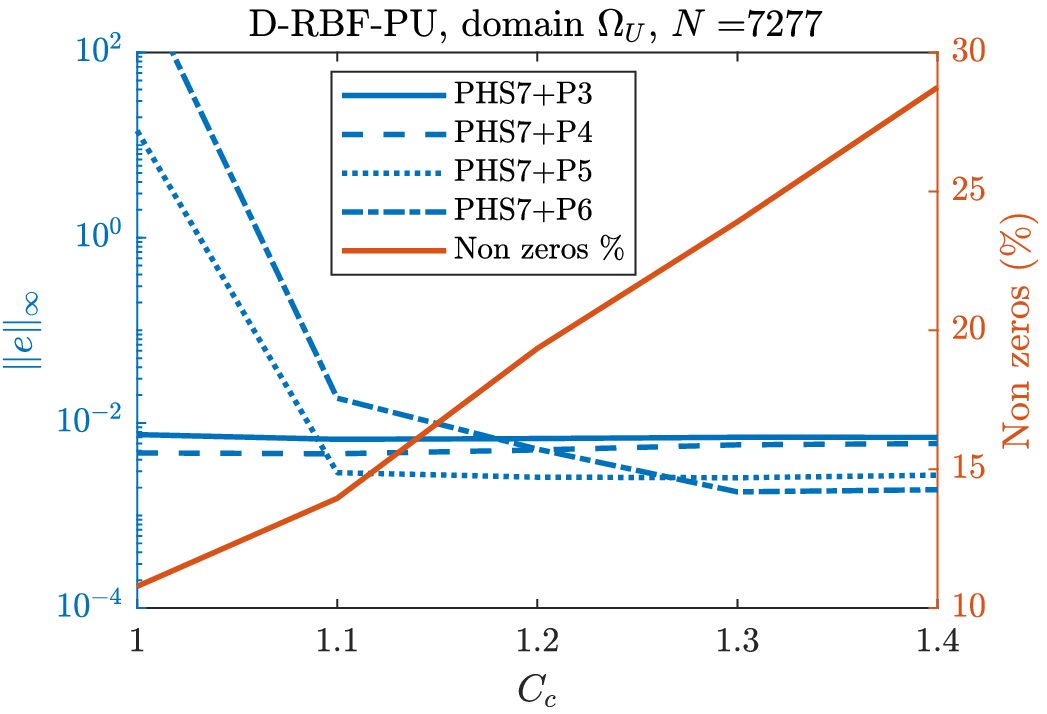}\includegraphics[width=5cm]{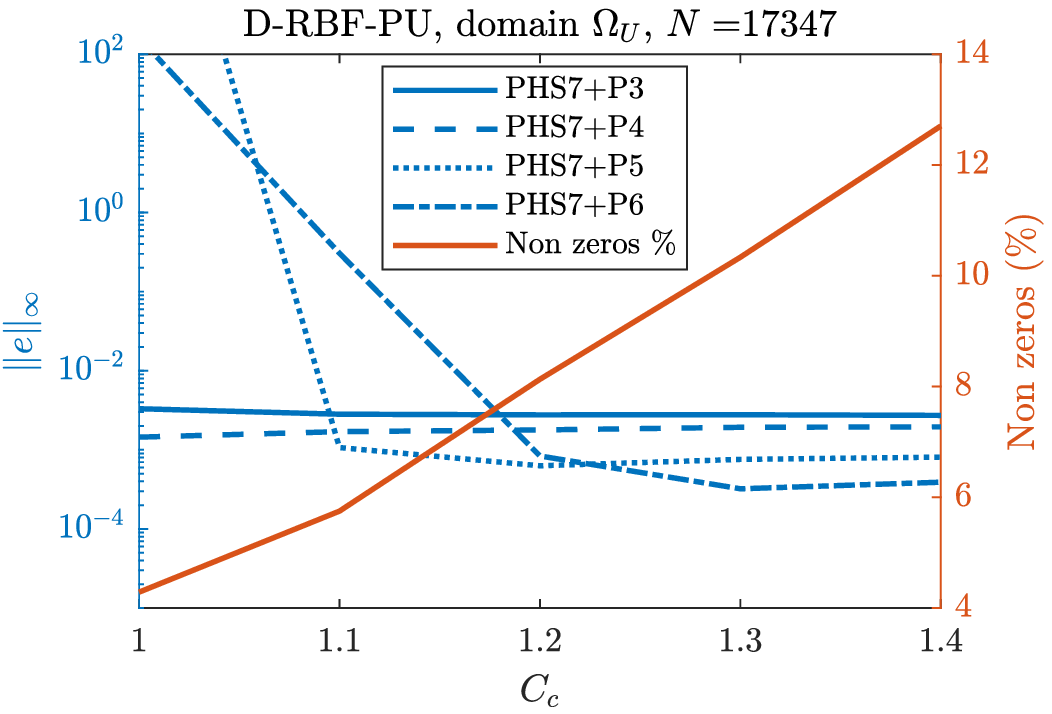}
\caption{\small{Accuracy (left vertical axis) and sparsity (right vertical axis) with respect to the overlap constant $C_c$. First row: PHS8 on 2D domain $\Omega_B$; second row: PHS7 on 3D domain $\Omega_U$. In all cases the smooth PU weight is applied.}}\label{fig_overlap1}
\end{center}
\end{figure}

To choose a proper overlap constant $C_c$, we illustrate some experiments in Figure \ref{fig_overlap1} for 2D (on $\Omega_S$) and 3D (on $\Omega_U$) problems. In this figure, errors (left vertical axis) and percentage of nonzero elements of the final matrix (right vertical axis) are plotted in terms of overlap constant $C_c$.
For smaller values of $C_c$ which are not covered in the plots, local RBF systems may not be full rank.
According to these and other experiments we use $C_c=1.0$ in all cases
except for 3D examples with polynomial spaces of order more than $6$ in which $C_c=1.2$ is used. This choices of the overlap constant make a balance between accuracy and sparsity.

\subsection{Convergence with respect to polynomial degrees}
In Figure \ref{fig_converg1}, the errors and convergence orders of the D-RBF-PU method with respect to the degree of polynomial spaces added to the RBF expansion are shown.
As the true solutions are infinitely smooth and the PDE is of the second order, the theoretical rate of convergence in all cases should be
${q-2}$ where $q$ is the order (degree $+\,1$) of appended polynomials. In most cases, this rate is achieved and in some cases we observe a higher convergence rate.
Corresponding plots for stability constant $\wt C_S$ are presented in Figure \ref{fig_converg2}. In all cases, no significant growth is observed as $N$ is increased. Obviously, this nice feature is inherited from the local property of the approximation method and is shared with other local methods such as RBF-FD.
{In this experiment we have used the smooth PU weight functions.}

\begin{figure}[!h]
\begin{center}
\includegraphics[width=4cm]{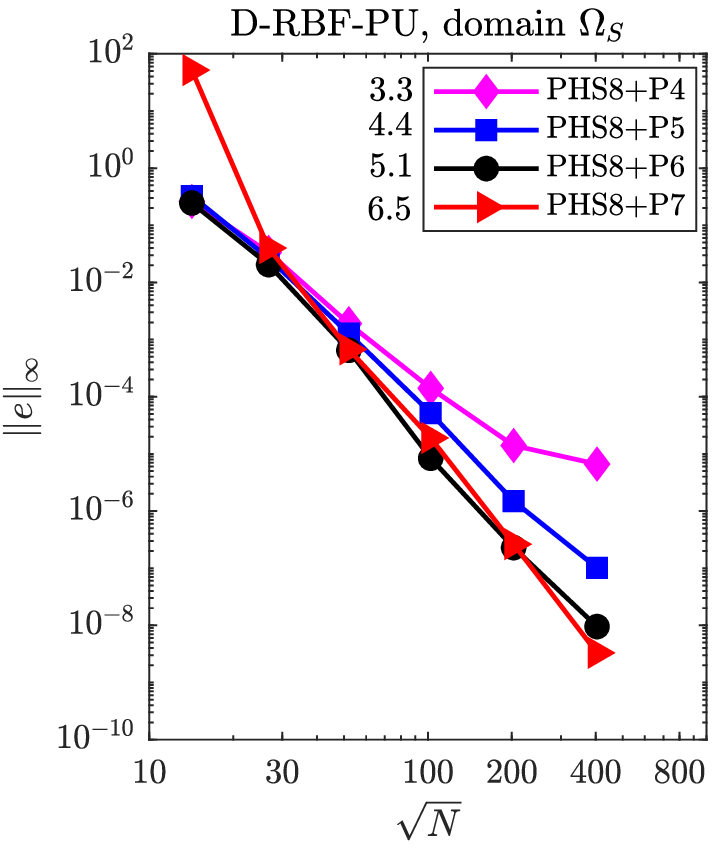}\includegraphics[width=4cm]{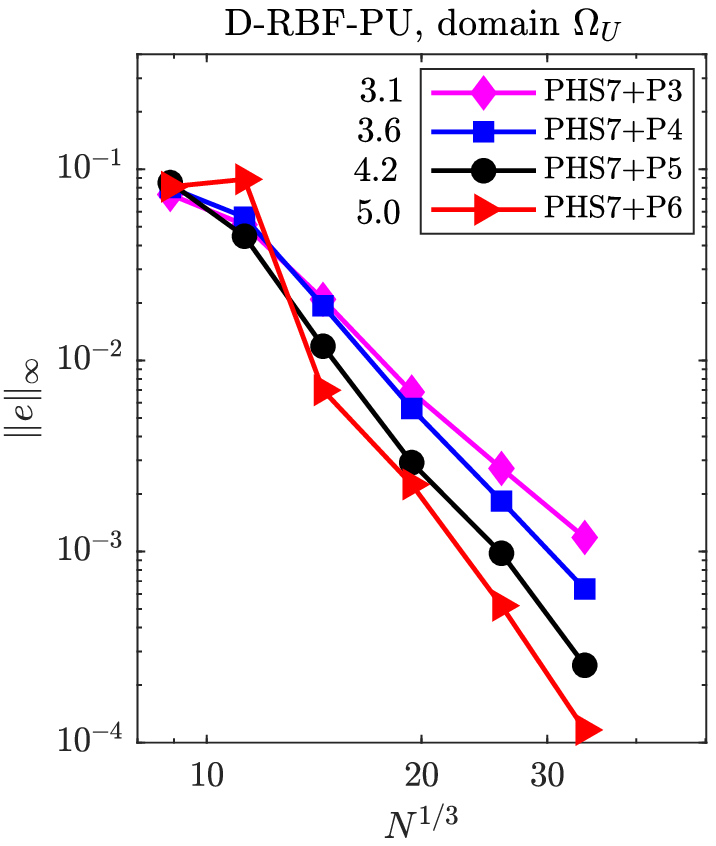}\includegraphics[width=4cm]{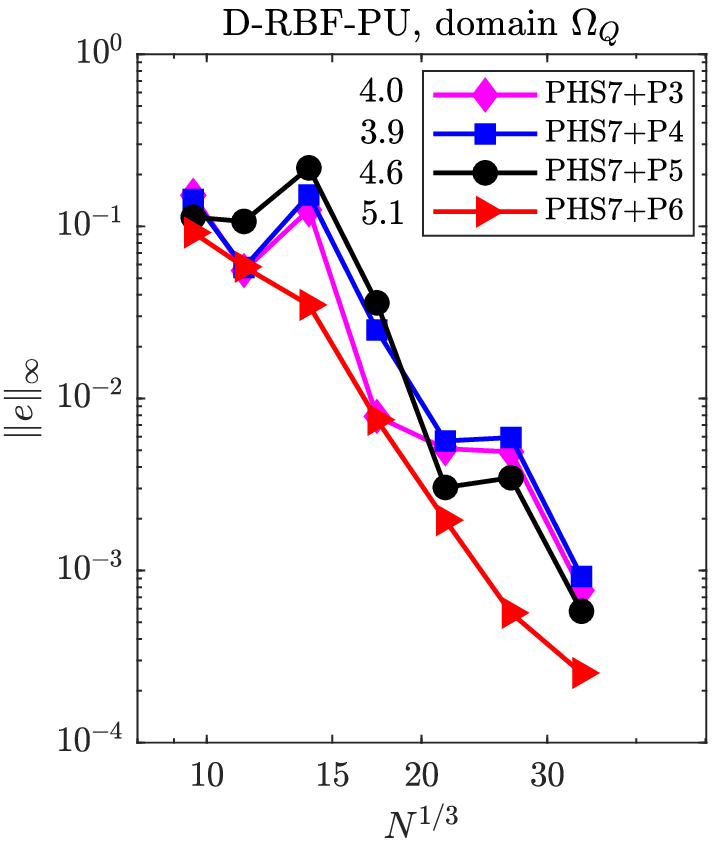}
\caption{\small{Errors and convergence orders of D-RBF-PU method with different polynomial degrees on domains $\Omega_S$ (left), $\Omega_U$ (middle) and $\Omega_Q$ (right). Theoretical orders are $q-2=\wt q-1$ where $\wt q$ is the degree of polynomial space. Here (for example) P3 means polynomial space of degree at most $3$. In all cases the smooth PU weight is applied.}}\label{fig_converg1}
\end{center}
\end{figure}

\begin{figure}[!h]
\begin{center}
\includegraphics[width=4cm]{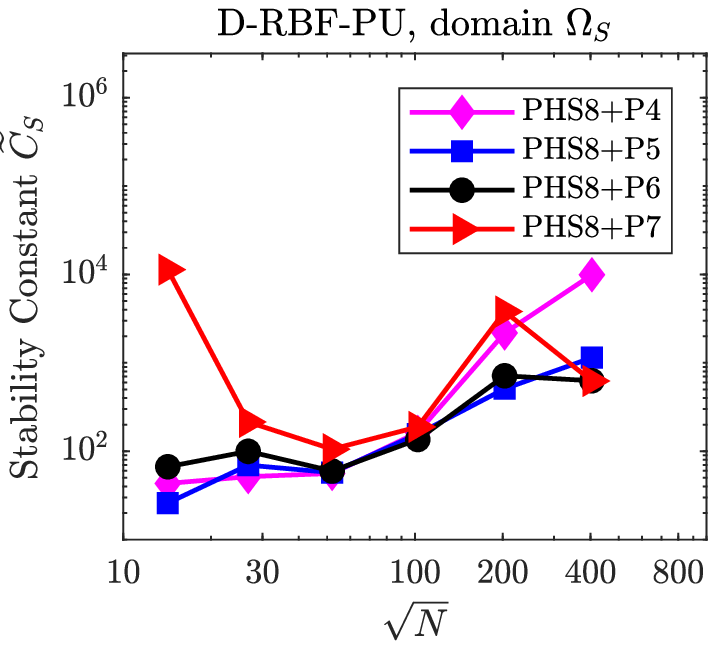}\includegraphics[width=4cm]{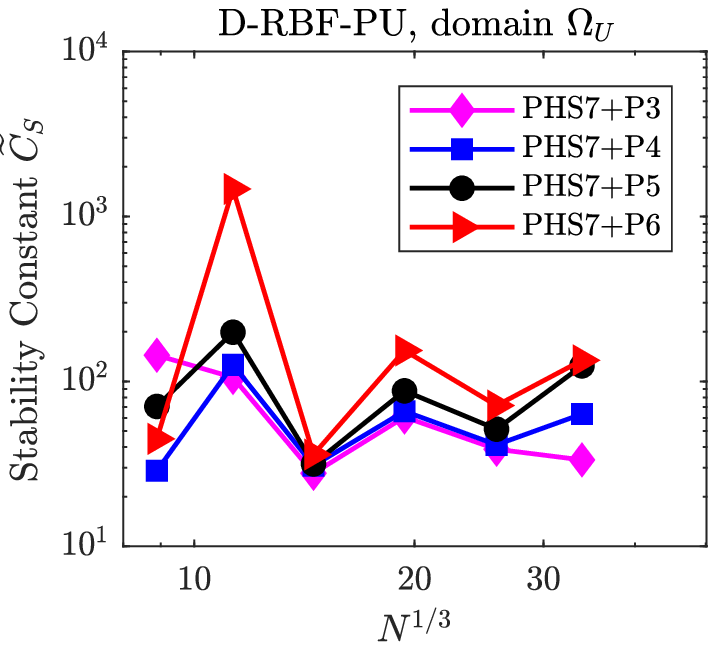}\includegraphics[width=4cm]{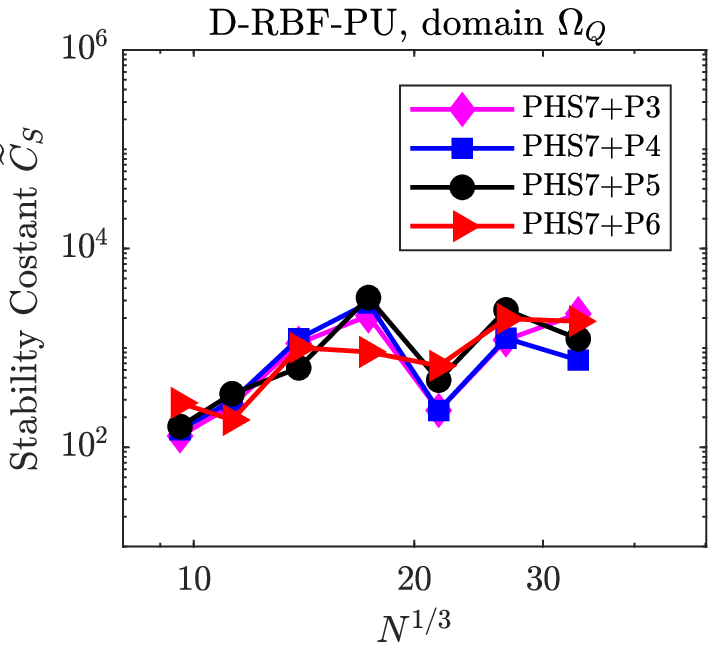}
\caption{\small{The stability constant $\wt C_S(A)$ of D-RBF-PU method with different polynomial degrees on domains $\Omega_S$ (left), $\Omega_U$ (middle) and $\Omega_Q$ (right). In all cases, no significant growth is observed as $N$ is increased. In all cases the smooth PU weight is applied.}}\label{fig_converg2}
\end{center}
\end{figure}

\subsection{Constant{-generated} PU weights}
Results for the constant{-generated} PU weight functions are reported in Figure \ref{fig_const1}.
In texts on figures, by `Const. Gen. PU Weight 1' and `Const. Gen. PU Weight 2' we mean weight functions \eqref{PUweight_const} and \eqref{PU-weight-const2}, respectively, and by `Hybrid PU Weights' we mean a combination of constant-{generated} weight \eqref{PU-weight-const2} for patches with centers inside $\Omega$ and smooth weight \eqref{wend_weight} for patches with centers on $\partial \Omega$.

Experiments show that in some cases with constant-generated PU weights to obtain the theoretical order, the number of patches and/or the size of patches (the overlap constant $C_c$) should be increased. This will increase the computational cost of the method. However,
since this lost of accuracy is caused by approximation at boundary points, a more efficient trick is to use the hybrid weight.
In Figure \ref{fig_const1} (right-hand side plots), an improvement in accuracies and an enhancement in numerical orders are observed by using the hybrid PU weight, where the smooth weight is used for boundary patches.

\begin{figure}[!h]
\begin{center}
\includegraphics[width=4cm]{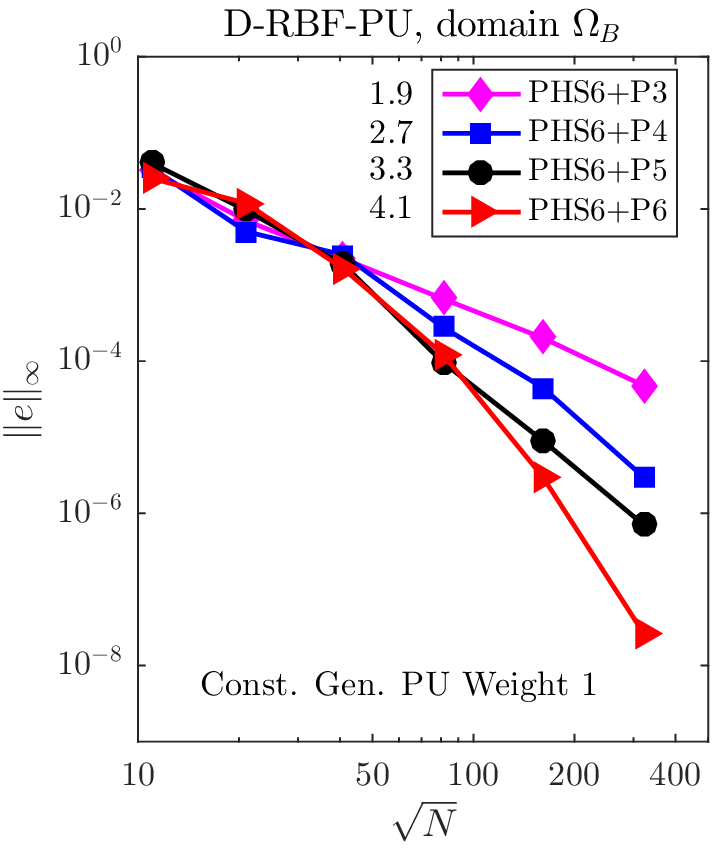}\includegraphics[width=4cm]{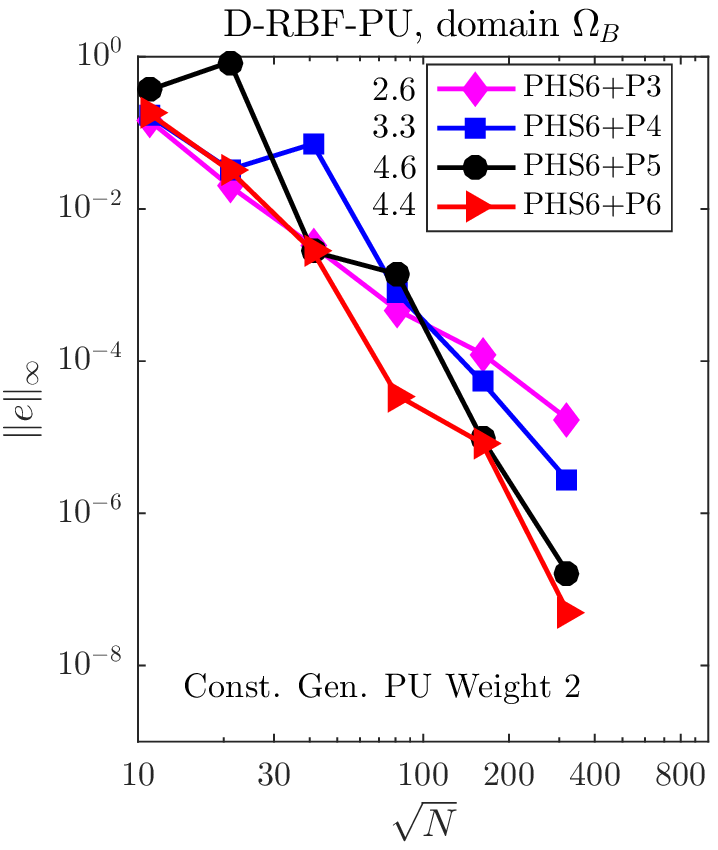} \includegraphics[width=4cm]{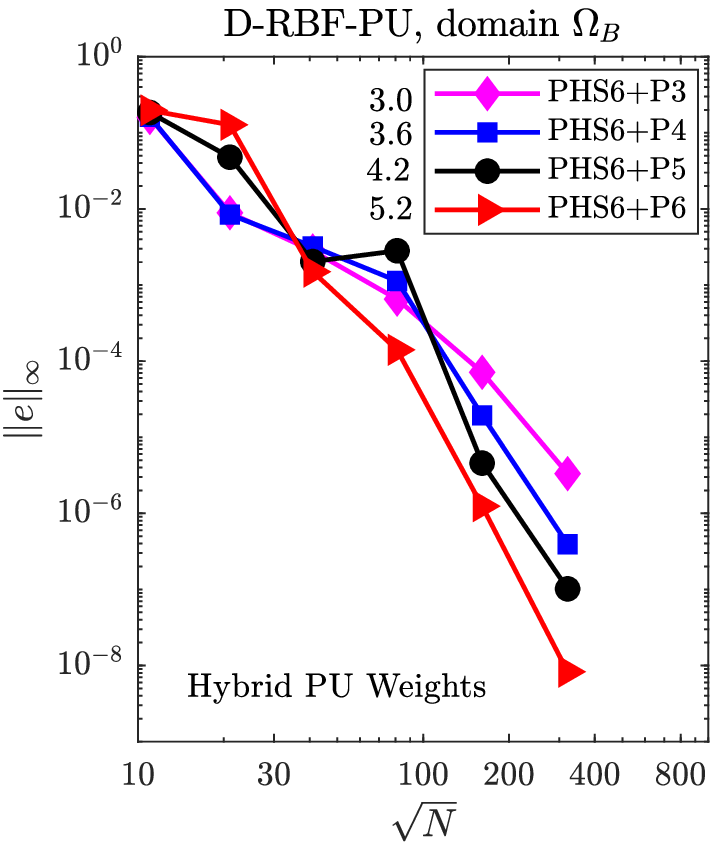}\\
\includegraphics[width=4cm]{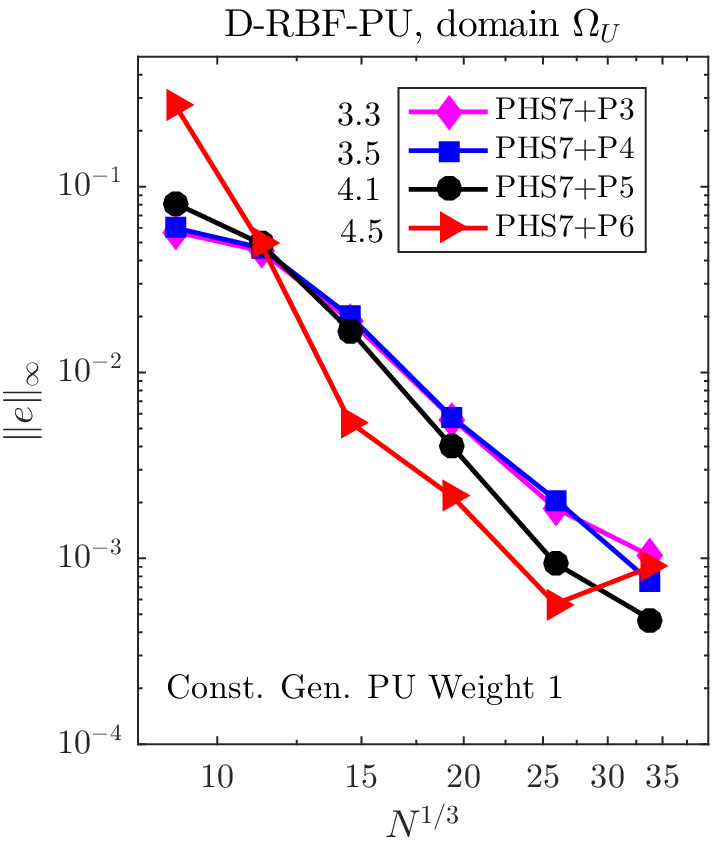}\includegraphics[width=4cm]{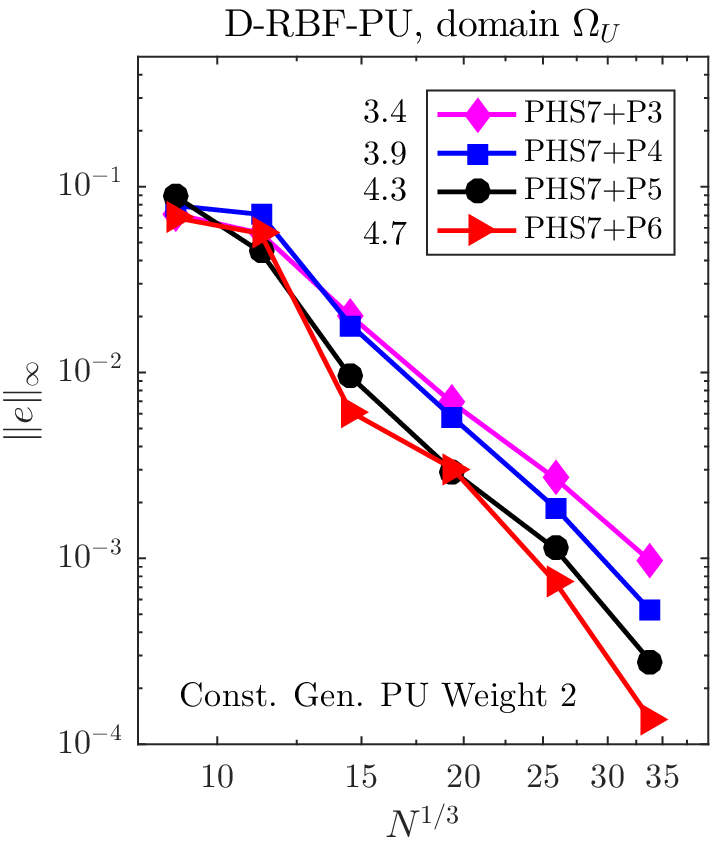} \includegraphics[width=4cm]{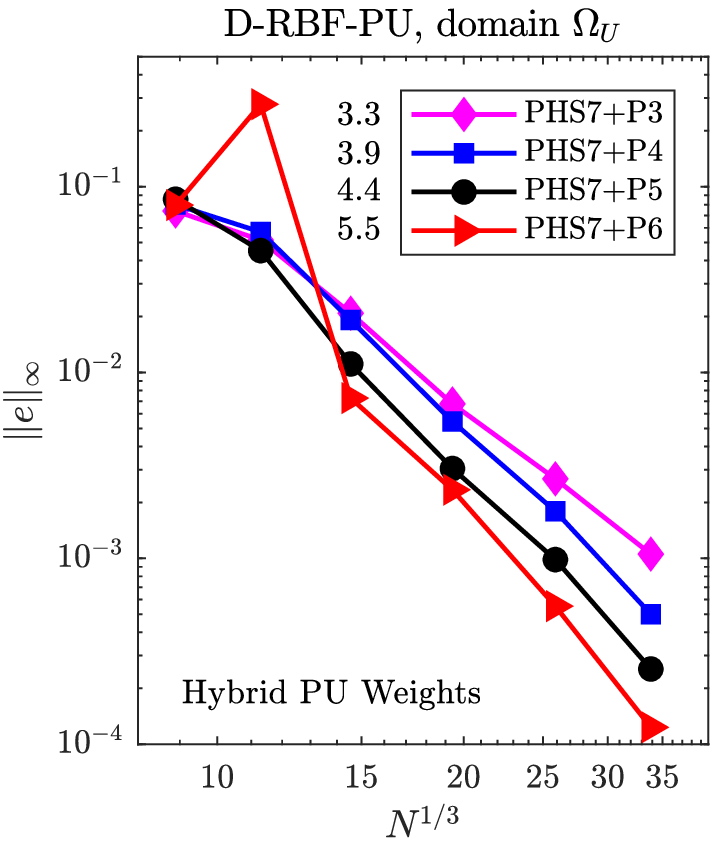}
\caption{\small{Errors and convergence orders of D-RBF-PU method with different polynomial degrees and with constant-generated PU weights on $\Omega_B$ (first row) and  $\Omega_U$ (second row). Theoretical orders are $\wt q-1$ where $\wt q$ is the degree of polynomial space.
Improvements in accuracies and orders are observed by the hybrid PU weight.
}}\label{fig_const1}
\end{center}
\end{figure}

Since, with the constant-generated PU weight \eqref{PU-weight-const2} the method uses a single patch for approximation at each test point, the resulting differentiation matrix is
the sparsest one. Combination with the smooth weight on the boundary does not increase the number of nonzeros, significantly, because the number of boundary points is of order $N^{1-1/d}$.
In Table \ref{tb1} the percentage of nonzero elements of the final matrix on 2D domain $\Omega_S$ and 3D domain $\Omega_U$ are given for three cases of weights and three different numbers of trial points.
In the 2D case, the number of nonzeros is nearly halved when the constant-generated weight is used instead of the smooth weight. The use of hybrid weight does not increase the percentage of nonzeros, remarkably, compared with the constant-generated weight.
In the 3D case, the constant-generated weight reduces the number of nonzeros by a factor of $\frac{1}{4}$, approximately.
Using the hybrid weight, the number of nonzeros is increased by a factor of $1.5$, approximately, but it is still far fewer
than that of the smooth weight. It is obvious that if in the 3D case we increase the number of trial points, the percentages
of the hybrid weight become closer to those of the constant-generated weight.

From the results of Table \ref{tb1} and the error plots of Figure \ref{fig_const1}, we may conclude that the D-RBF-PU method with the hybrid PU weight is a recommendable choice to obtain both reasonable accuracy and sparsity.

\begin{table}[!h]
\centering
\caption{The percentage of nonzero elements of the final matrix with three type of PU weights.}\label{tb1}
\begin{tabular}{|c|c|c|c|c|c|c}
  \hline
   &     & smooth  & const. gen.  & hybrid  \\
 domains  & $N$ &  weight &  weight \eqref{PU-weight-const2} &  weight \\
   \hline
                       &$2705$   &$4.95\,\%$  & $2.15\,\% $ & $2.38\,\% $ \\
  2D domain $\Omega_S$ &$10462$  &$1.15\,\%$  & $0.51\,\% $ & $0.54\,\%$\\
                       &$163554$ &$0.069\,\%$ & $0.031\,\%$ & $0.032\,\% $\\
\hline
                       & $7241$  & $23.4\,\%$   & $6.4\,\%$ &$9.7\,\%$ \\
  3D domain $\Omega_U$ & $17174$ & $10.4\,\%$ & $2.7\,\%$ &$3.9\,\%$ \\
                       & $38765$ & $4.4\,\%$  & $1.1\,\%$ &$1.6\,\%$ \\
\hline
\end{tabular}
\end{table}

The overall behaviour of stability plots for constant-generated weights is approximately the same as that of the smooth weight in Figure \ref{fig_converg2}.
Thus, we do not present them here to keep the total number of figures as low as possible.

\subsection{Comparison with standard RBF-PU}
{We compared the errors and orders of the D-RBF-PU and the standard RBF-PU methods verses
$N$ on different domains and by different kernels.
The accuracies of both methods are close
to each other so that in some cases the plots can not be easily distinguished.
The same holds also true for plots of the stability constants. However, to control the number of figures, we do not present them here.
Our observations} confirm the theoretical bounds of section \ref{sect-error-stability} and suggest to use the new method because it bypasses all derivatives of the PU functions and many lower derivatives of local approximants, while maintaining a similar accuracy and stability rate. Moreover, the new method allows to use a discontinuous weight function, a situation that cannot be treated by the standard RBF-PU method.

\subsection{Comparison with RBF-FD}\label{sect-cmpFD}
In order to compare the new method with RBF-FD, it is important to determine the size of stencils and their connection to the size of local patches in D-RBF-PU.
By the size of a stencil in RBF-FD, which is denoted by $\delta$ here, we mean the radius of the smallest ball that contains all the stencil points.
As same as the strategy we applied for near boundary patches, we increase the radius $\delta$ for test points on and close to the boundary (up to a radial $\delta$-distance from the boundary) by a factor of $1.5$ to have more accurate approximations near the boundary and, in particular, on boundary points.

\begin{figure}[!h]
\begin{center}
\includegraphics[width=4cm]{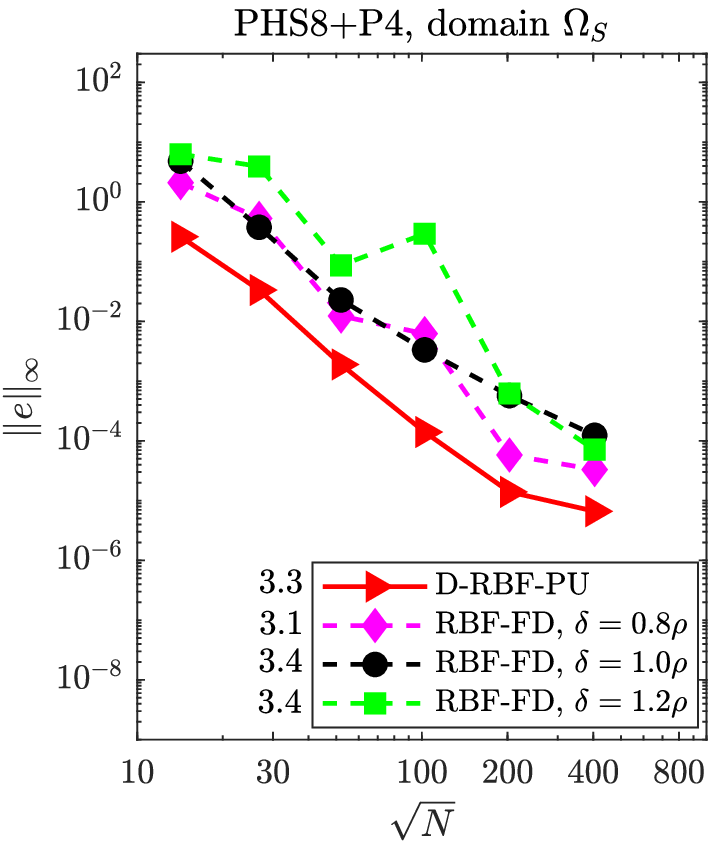}\includegraphics[width=4cm]{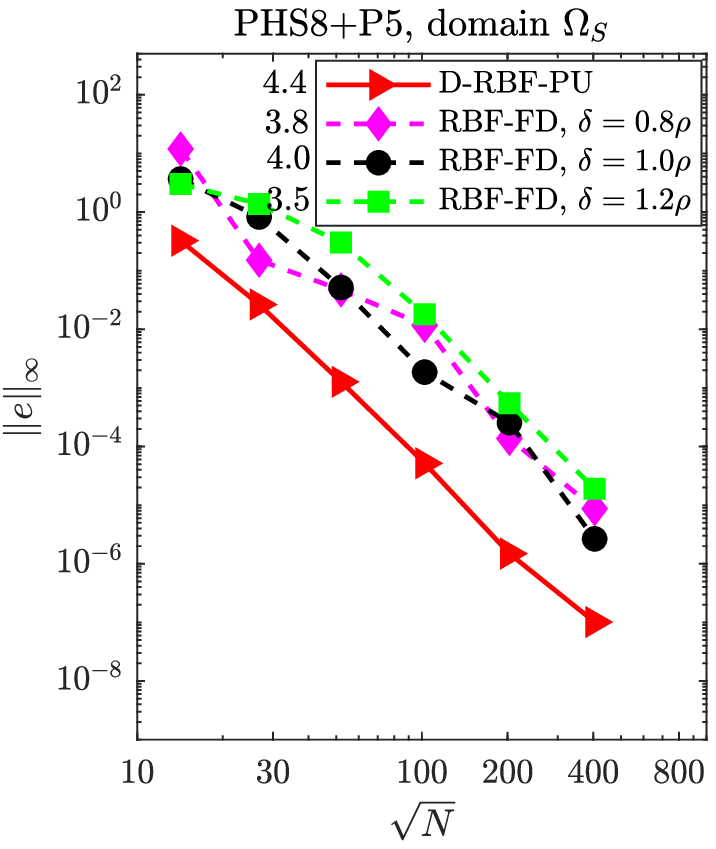}\includegraphics[width=4cm]{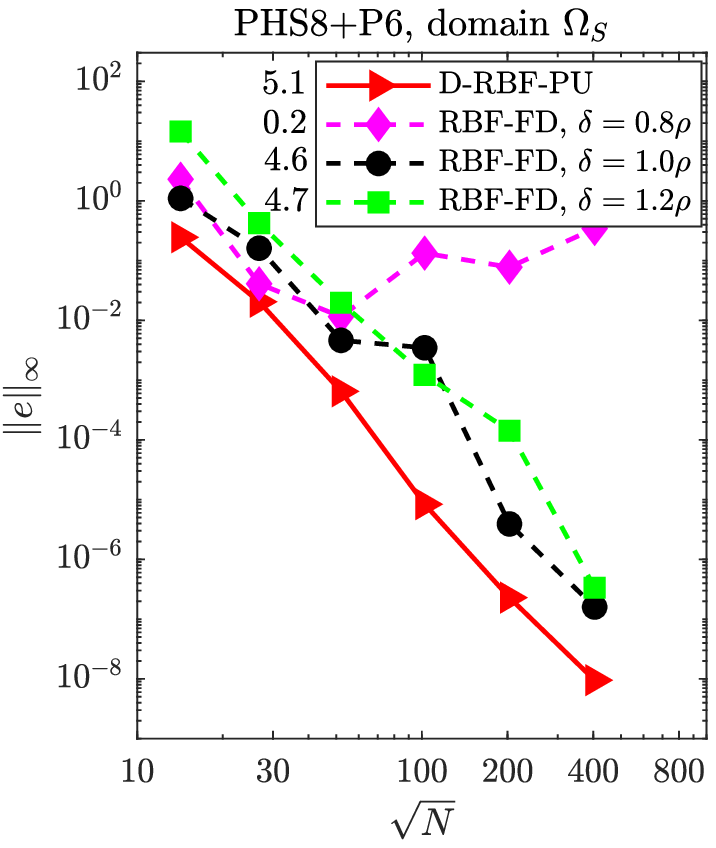}\\
\includegraphics[width=4cm]{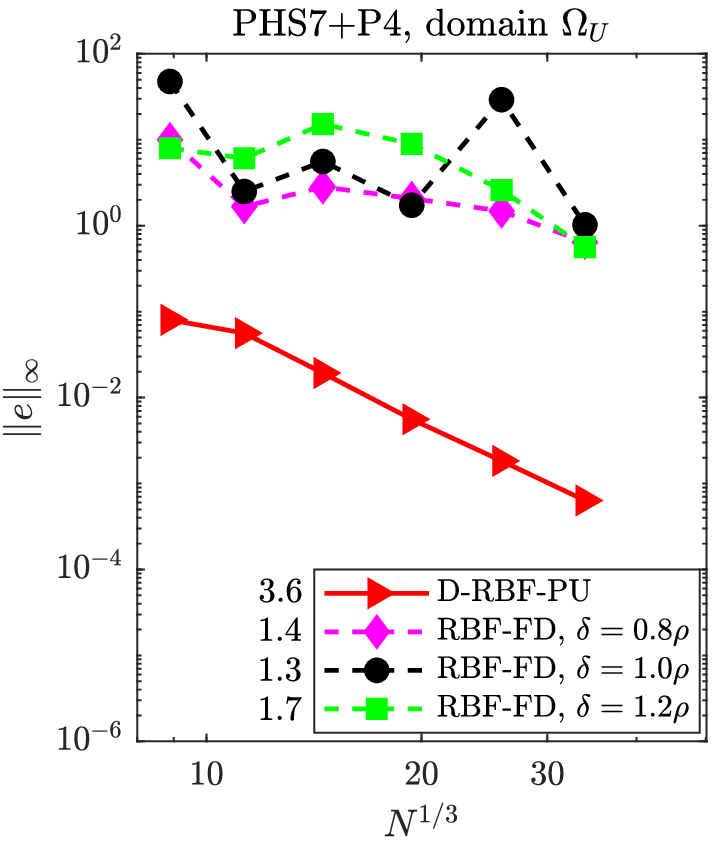}\includegraphics[width=4cm]{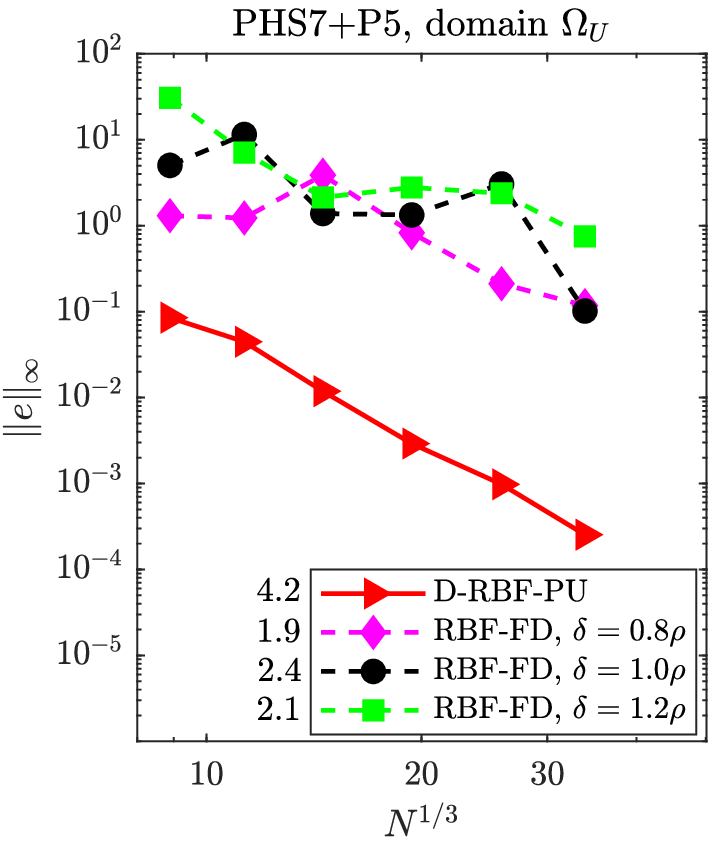}\includegraphics[width=4cm]{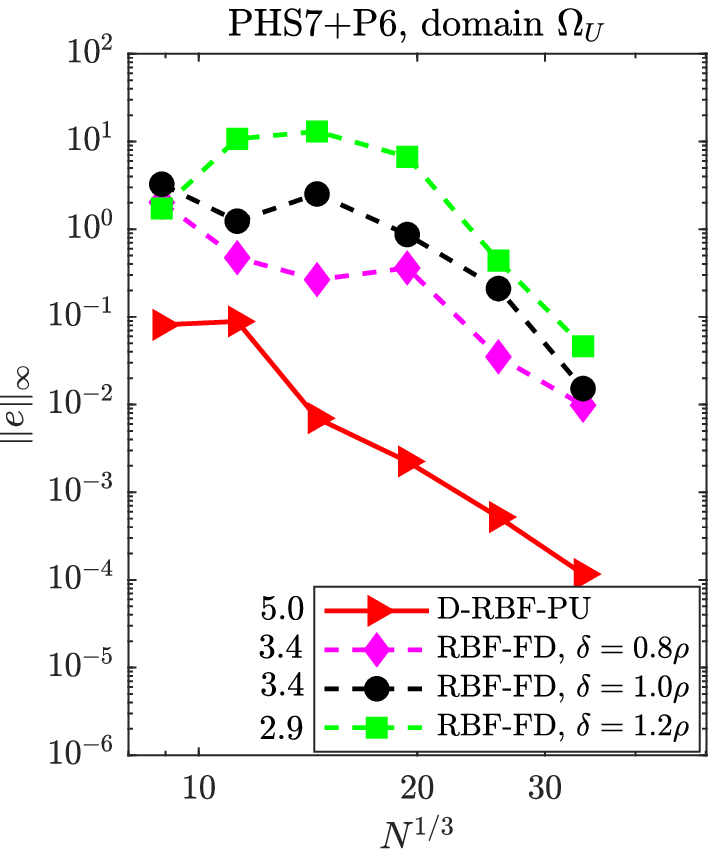}
\caption{\small{Errors and convergence orders of D-RBF-PU (using the smooth PU weight) and RBF-FD (with three different stencil sizes) on 2D domain $\Omega_S$ (first row) and 3D domain $\Omega_U$ (second row). Convergence orders and magnitude of errors are improved in the D-RBF-PU method. Here $\rho$ is the radius of covering patches and $\delta$ is the size of stencils in RBF-FD.}}\label{fig_cmpFD1}
\end{center}
\end{figure}

For each test point, RBF-FD uses a single stencil while D-RBF-PU uses much fewer number of covering patches and joins them
by PU functions.
If we compare with D-RBF-PU with a smooth weight function, we may assume
$
\wt X_k =\displaystyle\{\cup_{\ell\in J_k}X_\ell\} ,
$
where $X_\ell$ is the set of trial points in patch $\Omega_\ell$ and $\wt X_k$ is the stencil of test point $y_k$.
In this case, the same trial points contribute in approximation of $Lu(y_k)$ (or $Bu(y_k)$ if $y_k$ is a boundary point) in both methods.
However, a larger local RBF system \eqref{augmented-system} should be solved in RBF-FD while few (exactly $|J_k|$) number of much smaller systems \eqref{gen-lag-systm} (that may also be shared with other test points) need to be solved in D-RBF-PU. Experiments show that
in this case, the RBF-FD is very slow and its results are far away from the exact solutions because the size of stencils are overestimated. Hence, we do not illustrate the results in this case.
On the other hand, the size of RBF-FD stencils should be large enough to guarantee the solvability of local linear systems.
Here we compare both methods for three different stencil sizes $\delta = 0.8\rho, 1.0 \rho, 1.2\rho$ where $\rho$ is the radius of patches.
Smaller or bigger values of $\delta$ does not payoff for significantly more accurate results.
In Figure \ref{fig_cmpFD1}, we present the results on 2D domain $\Omega_S$ and 3D domain $\Omega_U$.
As we see, D-RBF-PU outperforms almost all cases especially in the 3D problem. While not presented here, the same is observed
on other domains.
In Figure \ref{fig_cmpFD2}, stability plots on $\Omega_U$ are illustrated. Both methods possess a nice stability that does not highly depend on increasing $N$ and polynomial degrees.
{The same behaviour is observed for the 2D case but not illustrated here.} However, in the 3D problem the stability numbers of RBF-FD are approximately $10^2-10^3$ times larger.
\begin{figure}[!h]
\begin{center}
\includegraphics[width=4cm]{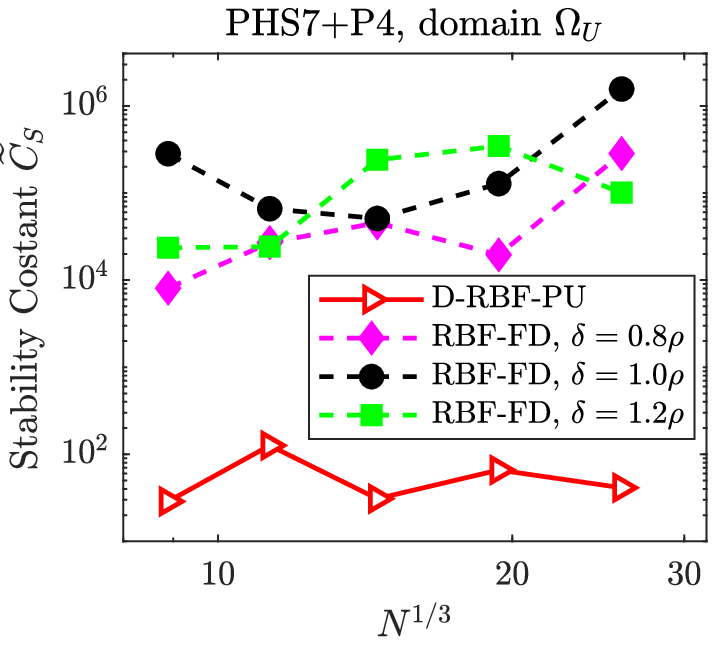}\includegraphics[width=4cm]{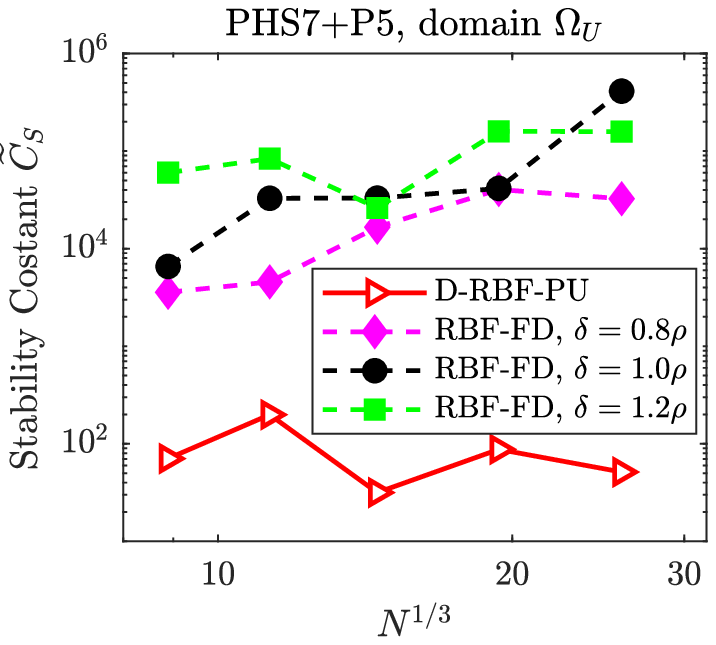}\includegraphics[width=4cm]{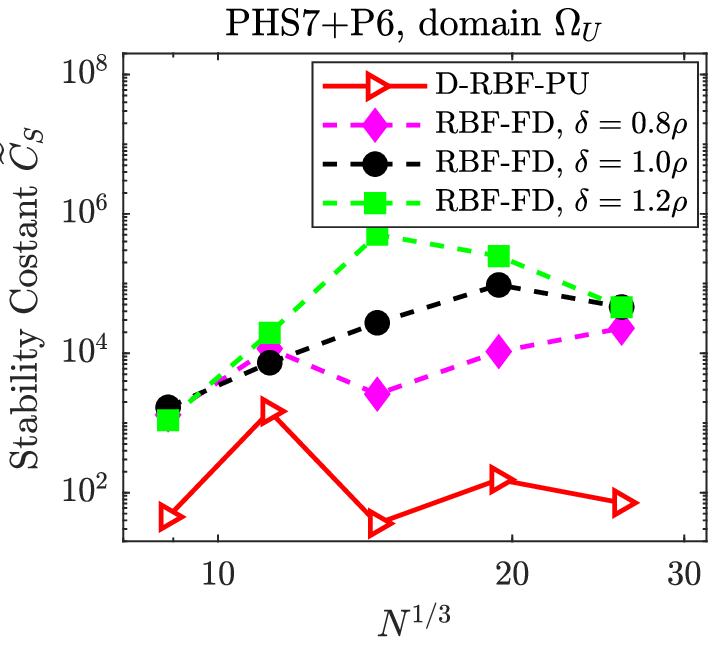}
\caption{\small{The stability constant $\wt C_S$ of global matrix $A$ of D-RBF-PU (using the smooth PU weight) and RBF-FD (with three different stencil sizes) on 3D domain $\Omega_U$. As $N$ is increased, no significant growth is observed in the conditioning of $A$ in both methods, although the stability numbers of RBF-FD are approximately $10^2-10^3$ times larger.}}\label{fig_cmpFD2}
\end{center}
\end{figure}

The above results were obtained for mixed Dirichlet and Neumann boundary conditions.
In Figure \ref{fig_cmpFD3} we show a comparison between the two methods on 3D domain $\Omega_Q$ when Dirichlet boundary condition is imposed on the whole boundary.
In this case, both RBF-FD and D-RBF-PU methods produce approximately the same error and convergence order. We observe the same results on other domains.
Comparing with previous figures, we conclude that
D-RBF-PU is more accurate than RBF-FD in the presence of Neumann boundary conditions. The reason seems to lie behind the fact that at a Neumann boundary point, in which
RBF-FD approximates derivatives by a single one-sided stencil, D-RBF-PU uses a weighted average approximation of several neighborhood patches.

\begin{figure}[!h]
\begin{center}
\includegraphics[width=4cm]{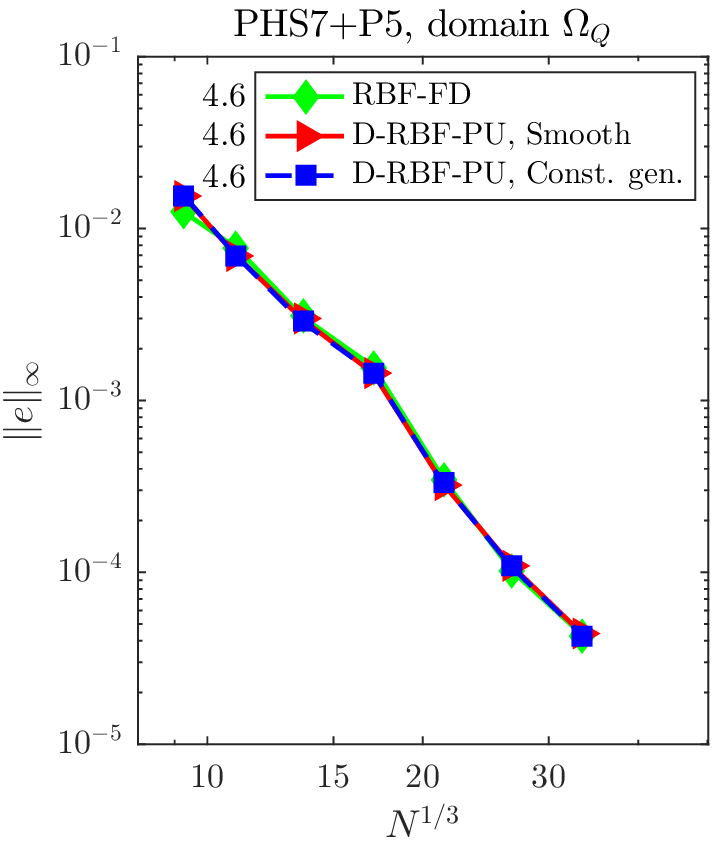}\includegraphics[width=4cm]{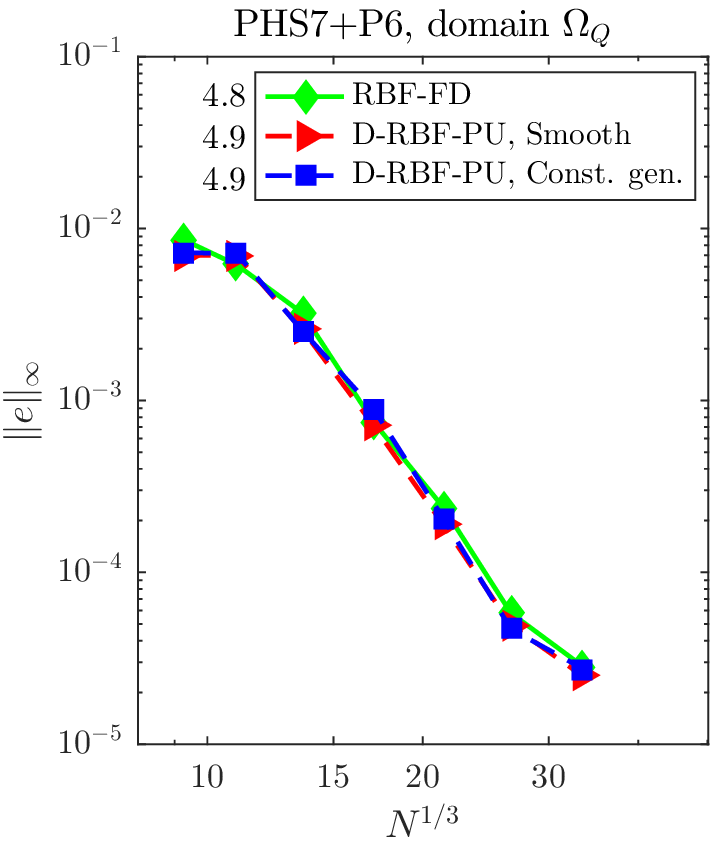}
\caption{\small{Errors and convergence orders of RBF-FD by stencil size $\delta=\rho$ and D-RBF-PU using the smooth and the hybrid PU weights on 3D domain $\Omega_Q$ with Dirichlet boundary conditions. We observe (approximately) the same errors and convergence orders.}}\label{fig_cmpFD3}
\end{center}
\end{figure}

\subsection{Time dependent PDEs}
In this section we focus on numerical solution of the diffusion equation
$$
\frac{\partial u}{\partial t} = \kappa \Delta u + f(x,t),\quad t>0, \; x\in \Omega\subset \R^d
$$
for $d=2,3$,
where $f(x, t)$ is a source term and $\kappa>0$ is the diffusion coefficient. Boundary condition
$Bu = g(x,t)$ for $x\in \partial \Omega$ should also be added, where $B$ is either Dirichlet and/or Neumann linear boundary operator.
We assign the prescribed solutions and apply the D-RBF-PU method to test spatial convergence
rates.

We apply method of lines (MOL) to time discretization, i.e., we approximate the spatial differential operators
with D-RBF-PU and then solve the resulting set of ordinary differential
equations (ODEs) using a backward time integrator. We assume $X = Y=Y_\Omega\cup Y_\Gamma$, $\u = [\u_\Omega; \u_\Gamma]$,
$\f_\Omega = f|_{Y_\Omega}$, $\g_\Gamma = g|_{Y_\Gamma}$,
$A_L = [A_{\Omega\Omega}~A_{\Omega\Gamma}]$ and $A_B = [A_{\Gamma\Omega}~A_{\Gamma\Gamma}]$ to get
$$
\begin{bmatrix}
\frac{\partial \u_\Omega}{\partial t}(t) \\ 0
\end{bmatrix} =
\begin{bmatrix}
\kappa A_{\Omega\Omega} & \kappa A_{\Omega\Gamma} \\
 A_{\Gamma\Omega} &  A_{\Gamma\Gamma}
\end{bmatrix}
\begin{bmatrix}
\u_\Omega(t) \\ \u_\Gamma(t)
\end{bmatrix} +
\begin{bmatrix}
\f_{\Omega}(t) \\ -\g_{\Gamma}(t)
\end{bmatrix} .
$$
Let $t_{k+1}=t_k+\Delta t$ where $\Delta t$ is the time step, and $k$ indexes a time level. Using superscripts
for time levels, the one order backward differentiation formulae (BDF1) reads as
\begin{equation*}
\begin{bmatrix}
I-\kappa\Delta t A_{\Omega\Omega} & -\kappa\Delta t A_{\Omega\Gamma}\\
A_{\Gamma\Omega} &  A_{\Gamma\Gamma}
\end{bmatrix}
\begin{bmatrix}
\u_\Omega^{k+1} \\ \u_\Gamma^{k+1}
\end{bmatrix} =
\begin{bmatrix}
\u_\Omega^{k}+\Delta t\f_{\Omega}^{k+1} \\ \g_{\Gamma}^{k+1}
\end{bmatrix}.
\end{equation*}
For a pure Dirichlet boundary condition we have $A_{\Gamma\Omega} = 0$ and $A_{\Gamma\Gamma} = I$ that reduce the system 
to
\begin{equation*}
(I-\kappa \Delta t A_{\Omega\Omega}) \u_\Omega^{k+1} = \u_\Omega^k + \Delta t (\f_\Omega^{k+1} + \kappa A_{\Omega\Gamma}\g_\Gamma^{k+1}).
\end{equation*}
The stability region of BDF1 is $S_1:=\{\lambda\in \mathbb C: |\lambda-1|\geqslant 1\}$. Since the region $S_1$ contains the left half of the complex plane, BDF1 is an A-stable ODE solver. In Figure \ref{fig_heat1} the spectrum of the discrete Laplacian with the new method is shown for 2D and 3D cases, respectively.

\begin{figure}[!h]
\begin{center}
\includegraphics[width=6cm]{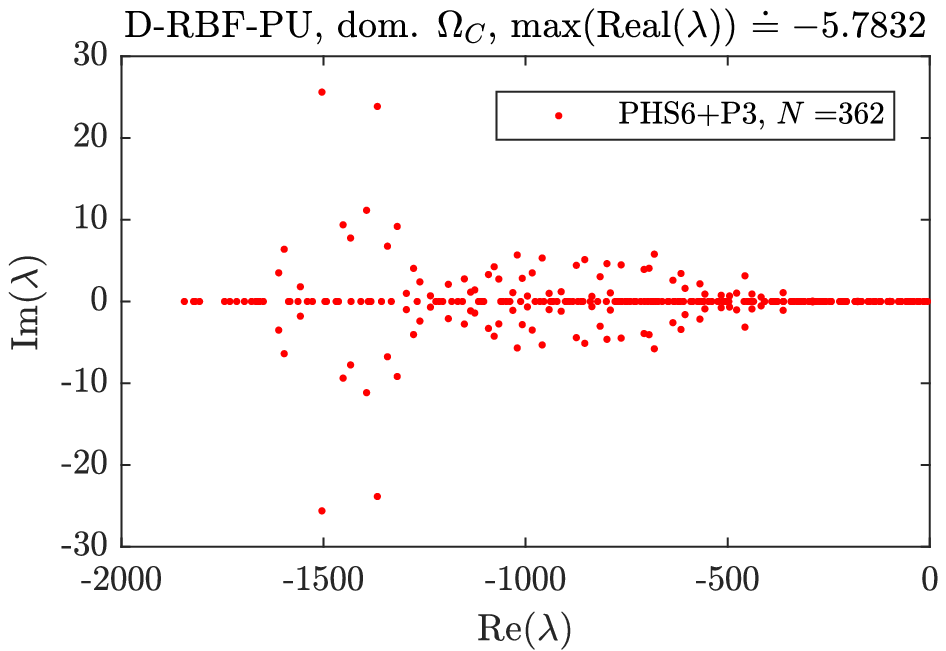}\includegraphics[width=6cm]{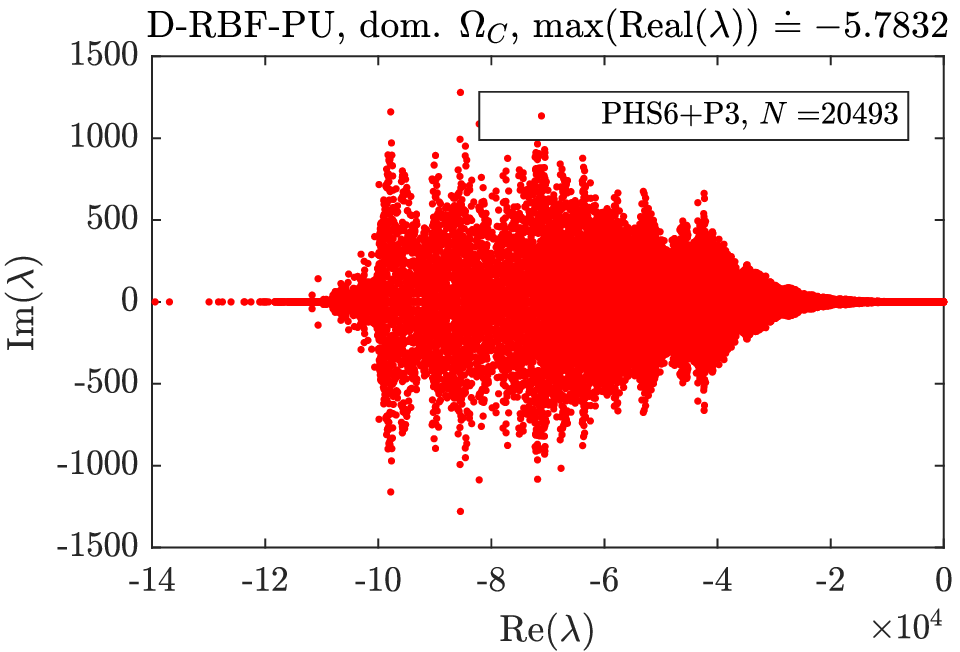}\\
\includegraphics[width=6cm]{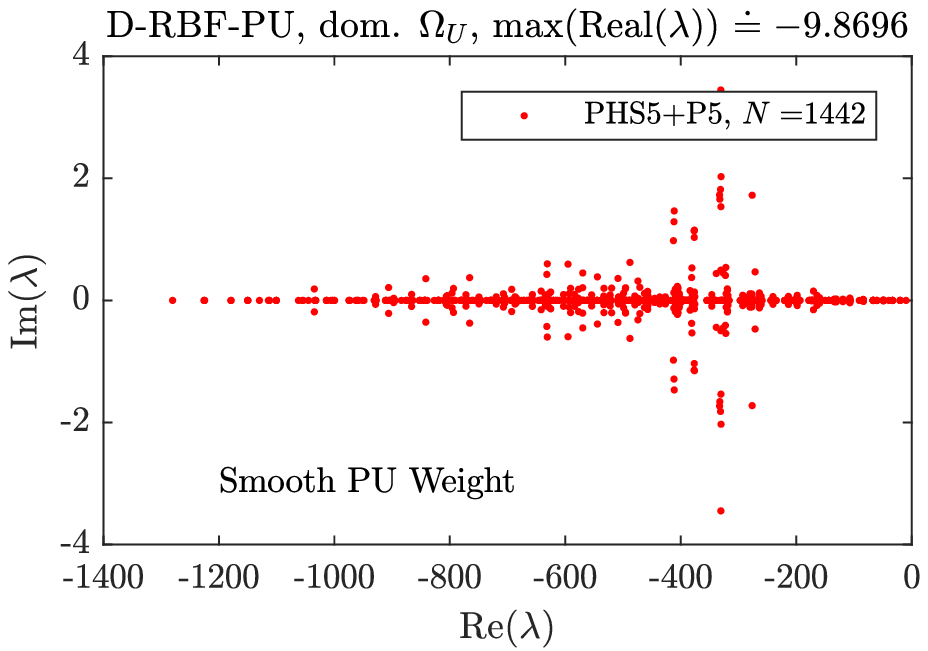}\includegraphics[width=6cm]{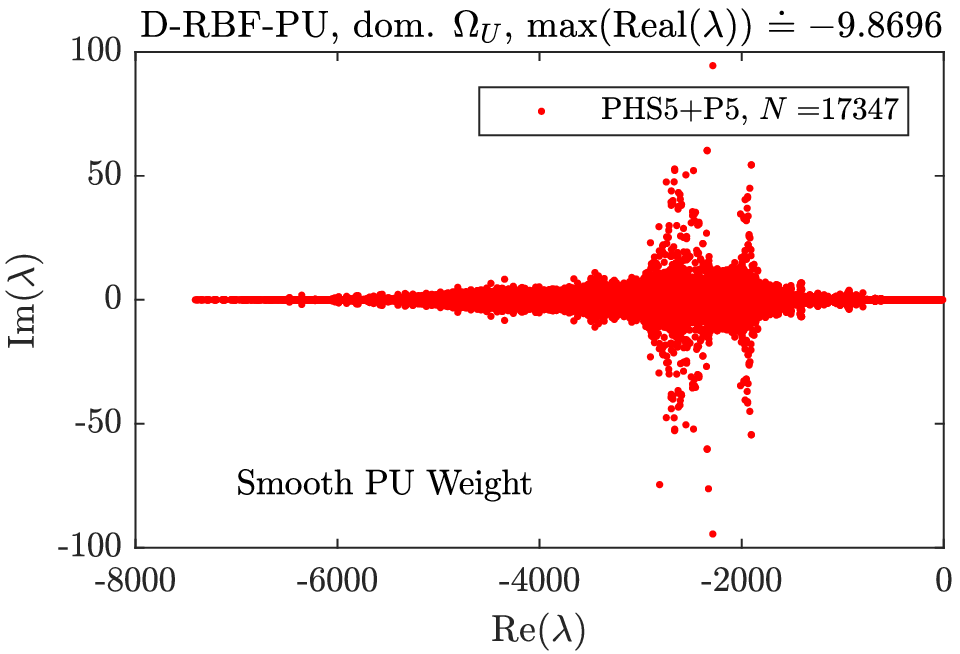}
\caption{\small{The spectrum of the discrete Laplacian with D-RBF-PU method on 2D domain $\Omega_C$ (first row) and 3D domain $\Omega_U$ (second row). All eigenvalues fall on the left half plane. In 2D,  the case PHS6+P3 and in 3D, the case PHS5+P5 are depicted. In other cases (alternative set of points or polynomial degrees), the spectrum enjoys similar features.}}\label{fig_heat1}
\end{center}
\end{figure}

In all cases the entire spectrum falls in the left half
plane. Though not illustrated here, similar behaviors are observed for other PHS kernels and higher degree polynomials.
The maximum real part of eigenvalues
is written in the title of figures; in all cases it is far away from origin allowing to use higher order backward differentiation formulas as their stability region exclude small parts near the imaginary axis.  Here we use the BDF4 scheme
\begin{equation*}
\begin{bmatrix}
I-\frac{12}{25}\kappa\Delta t A_{\Omega\Omega} & -\frac{12}{25}\kappa\Delta t A_{\Omega\Gamma}\\
A_{\Gamma\Omega} &  A_{\Gamma\Gamma}
\end{bmatrix}
\begin{bmatrix}
\u_\Omega^{k+1} \\ \u_\Gamma^{k+1}
\end{bmatrix}
=\begin{bmatrix}
R(\u_\Omega^{k},\u_\Omega^{k-1},\u_\Omega^{k-2},\u_\Omega^{k-3})+\frac{12}{25}\Delta t\f_{\Omega}^{k+1} \\ \g_{\Gamma}^{k+1}
\end{bmatrix}.
\end{equation*}
where
$$R(\u_\Omega^{k},\u_\Omega^{k-1},\u_\Omega^{k-2},\u_\Omega^{k-3}) = \frac{48}{25}\u_\Omega^{k}-\frac{36}{25}\u_\Omega^{k-1}+\frac{16}{25}\u_\Omega^{k-2}-\frac{3}{25}\u_\Omega^{k-3}.
$$
For a pure Dirichlet boundary condition the above system is reduced
to
\begin{equation*}
\Big(I-\frac{12}{25}\kappa \Delta t A_{\Omega\Omega}\Big) \u_\Omega^{k+1} = R(\u_\Omega^{k},\u_\Omega^{k-1},\u_\Omega^{k-2},\u_\Omega^{k-3}) + \frac{12}{25}\Delta t (\f_\Omega^{k+1} + \kappa A_{\Omega\Gamma}\g_\Gamma^{k+1}).
\end{equation*}
Experimental results show that both systems are invertible.

{We start the BDF4 scheme with the exact values for the first three time steps.
However, in a practical problem when the true solution is unknown we can start the BDF4 scheme with a step each of BDF1, BDF2 and BDF3.}
In all cases, we set $\kappa=1$ and $\Delta t = 0.005$, and report the results at the final time $t=0.2$.

To solve the time-stepping linear system, inspired by \cite{shankar:2017-1}, the GMRES
method using the incomplete LU preconditioner with a drop tolerance of $10^{-8}$ is used.
In our experiments, GMRES scheme converges at each time step in 2-3 iterations without restart for a relative residual less than a prescribed tolerance of
$10^{-12}$.

We consider the pure Dirichlet and the mixed Dirichlet-Neumann problems. In Figures \ref{fig_heat3}, the error plots of the RBF-FD and the D-RBF-PU methods are
given for the 2D problem on $\Omega_C$ and the 3D problem on $\Omega_U$ \cite{shankar:2017-1}. For the Dirichlet problem,
both methods converge and attain the theoretical orders.
For the Neumann-Dirichlet problems, results of the standard RBF-FD method are quickly blown up when advancing in time. Even the BDF1 scheme fails to fix this problem. However, the D-RBF-PU method converges for both 2D and 3D problems as errors and convergence orders with the hybrid PU weight are shown in Figure
\ref{fig_heat3} (right-hand side plots). The same result is obtained by the smooth PU weight, which is not illustrated here.

However, the new method, as well as the RBF-FD method, fails to give accurate results with BDF schemes for the Neumann or Neumann-Dirichlet problems on the 3D domain $\Omega_Q$ with irregular trial points. This problem might be alleviated by other improvements such as adding an extra set of points inside the domain adjacent to the boundary \cite{shankar:2017-1,shankar-fogelsoin:2018-1} or using a set of ghost
points outside the domain boundary \cite{flyer-et-al:2016-1}. We do not pursue this further and leave it for an independent study.

\begin{figure}[!h]
\begin{center}
\includegraphics[width=4cm]{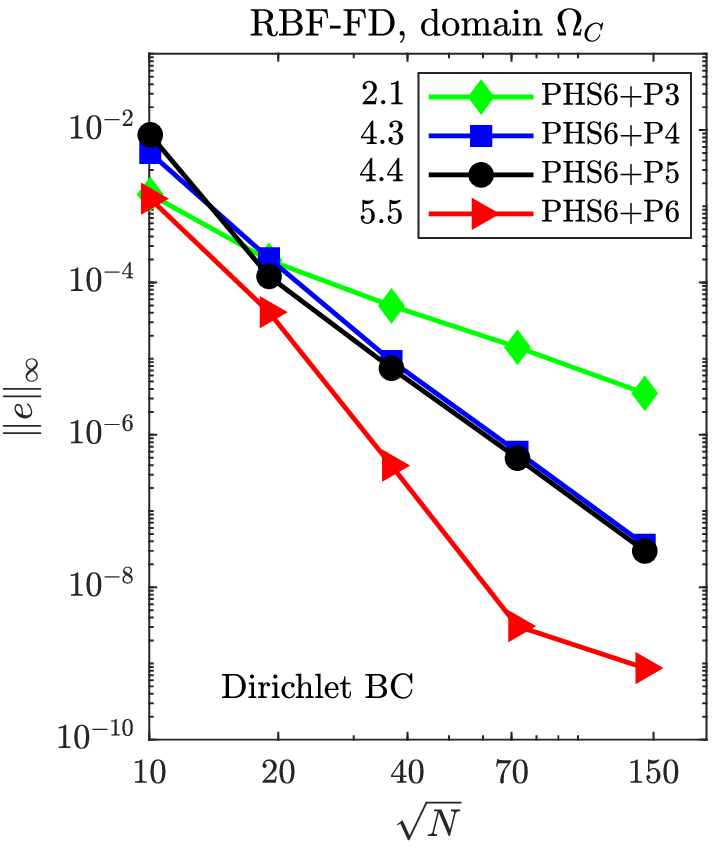}\includegraphics[width=4cm]{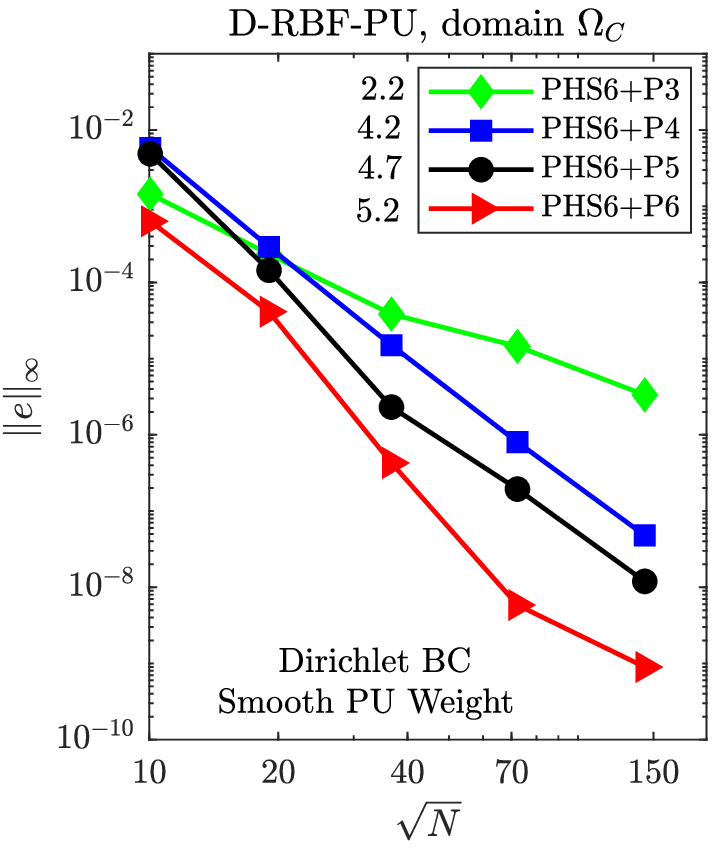}\includegraphics[width=4cm]{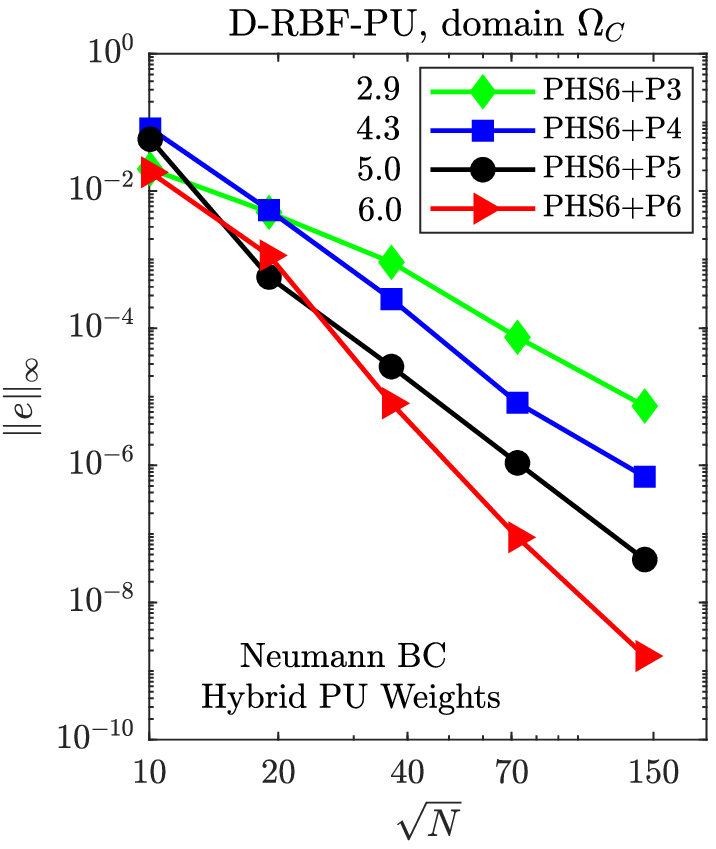}\\
\includegraphics[width=4cm]{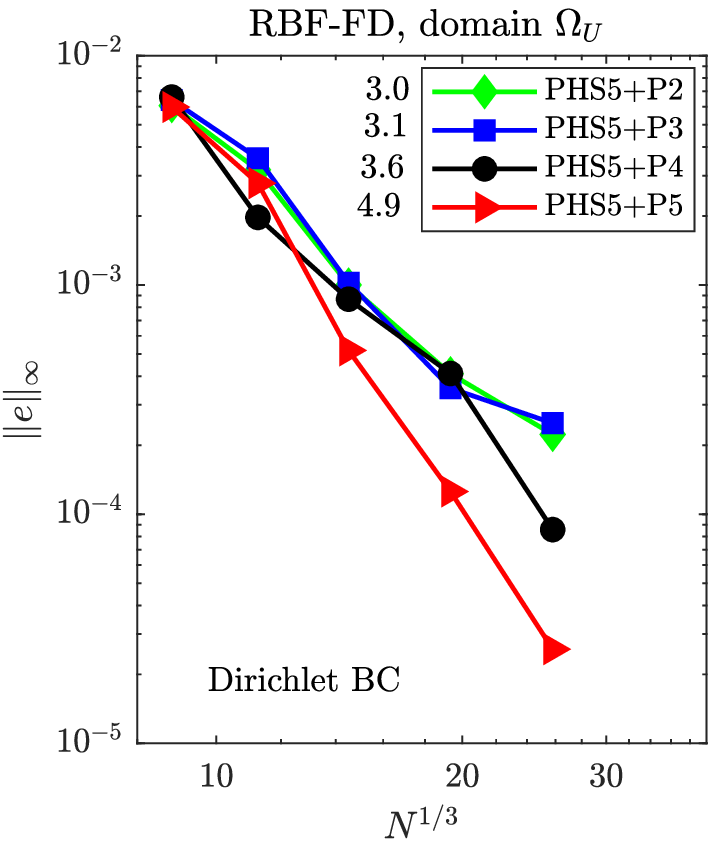}\includegraphics[width=4cm]{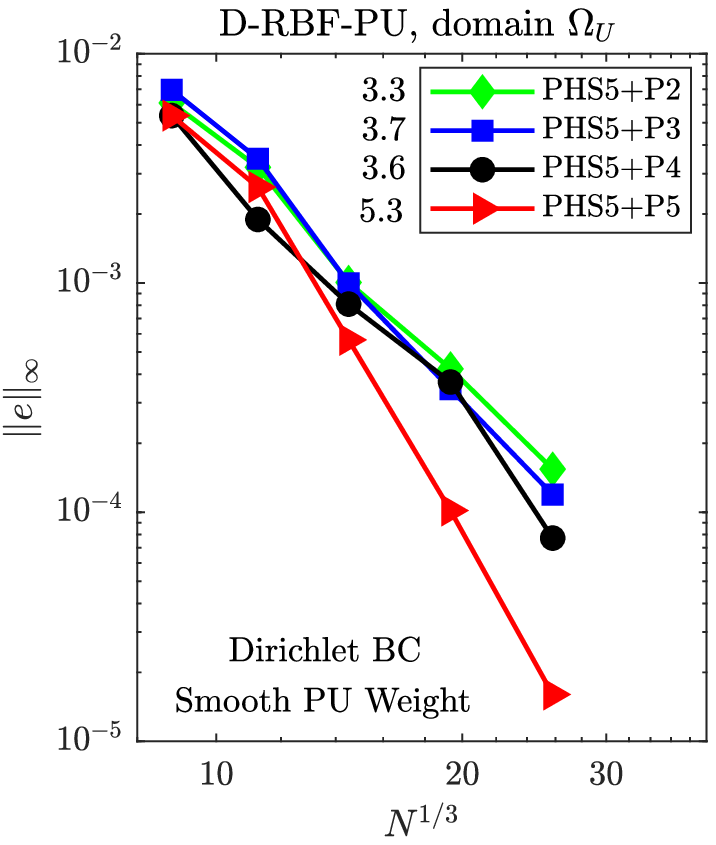}\includegraphics[width=4cm]{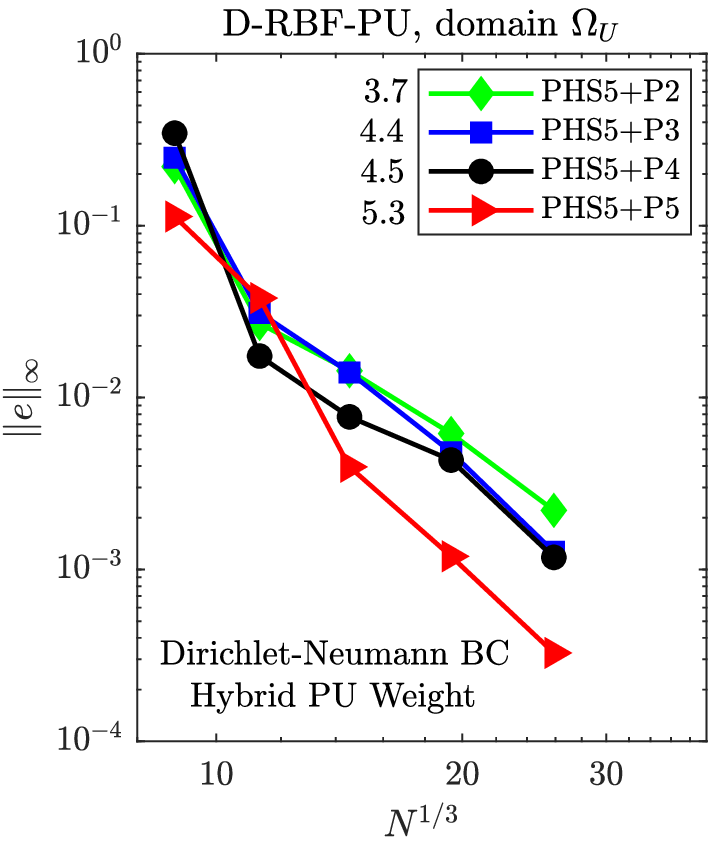}\\
\caption{\small{Errors and convergence orders of the RBF-FD (left) and the D-RBF-PU (middle) methods for heat equation on 2D domain $\Omega_C$ (first row) and 3D domain $\Omega_U$ (second row) with Dirichlet boundary conditions. The right-hand side plots show the errors and convergence orders of the D-RBF-PU method when Neumann-Dirichlet boundary conditions are imposed.
Theoretical orders are $\wt q-1$} where $\wt q$ is the degree of appended polynomials. In all cases, the observed convergence orders are better than those predicted by theoretical bounds.}\label{fig_heat3}
\end{center}
\end{figure}

\subsection{Computational costs}
In RBF-FD, the stencil $\wt X_k$ is changed per test point $y_k$, and
if $M$ is the number of total test points, $M$ local linear systems should be solved for setting up the global matrix $A$.
In D-RBF-PU, this number is highly reduced to $N_c$; the number of PU patches.
Remember that the number of patches is much smaller than the number of test points; if $h_c=\alpha h$ then in the square case $M = N \approx \alpha^d N_c$. For example, in Figure \ref{fig_covering1} for the 2D domain $\Omega_S$ we have $681$ test points while there are $44$ patches. Or, for a case in the 3D domain $\Omega_Q$, we may have $10078$ test points verses $208$ patches.
This leads to a remarkable difference between the computational costs of RBF-FD and D-RBF-PU methods for constructing the final differentiation matrix.
More precisely, D-RBF-PU should be approximately $\alpha^d$ times faster than RBF-FD at solving local linear systems. In our experiments with
$\alpha=4$, we should obtain speedups of approximately 16x and 64x for solving local problems in 2D and 3D, respectively.
But, D-RBF-PU (with a smooth weight) has its own cost for computing the PU weights and joining the local approximants. These may reduce
the above speedups for the total cost of the setup phase.
On the other side, since more points are contributed in PU approximations, the final matrix of the RBF-FD is sparser than that of the D-RBF-PU with a smooth PU weight. This makes RBF-FD faster at solving the final system. However, the D-RBF-PU with the constant-generated weight \eqref{PU-weight-const2} is as sparse as RBF-FD, because it uses a single patch for any test point.
Here, we compare the CPU times for both setting up the final matrix and solving the final system in terms of $N$, the number of trial points. RBF-FD with $\delta=\rho$ is considered.
Results are given for a 2D and a 3D problem in {Figure \ref{fig_cpu}. The total time panel is not given in the 2D case because the CPU times for solving the final systems are neglectable compared with those for the setting up the matrices. Although in this logarithmic scale the three lines of the D-RBF-PU method are close to each other, by looking at the numbers for large values of $N$, we find that in the setting up phase of the 2D case}, D-RBF-PU is approximately $5$, $9$ and $10$ times faster than RBF-FD for smooth, hybrid and constant-generated weights, respectively. These factors are increased to more than $20$, $30$ and $40$, respectively, for the 3D case. However, the solving times for the 3D systems are increased as they are denser than the 2D systems. Comparing the total times for large values of $N$, we again observe speedups of more than 5x, 9x and 10x for smooth, hybrid and constant-generated weights, respectively.
Thanks to the GMRES algorithm with incomplete LU factorization, these speedups remain (approximately) unchanged for the heat equation.
As a consequence, the hybrid case in D-RBF-PU is recommended as a first choice because it possesses both high accuracy and low complexity, simultaneously.

{Although the new method is more accurate in presence of the Neumann boundary conditions (as we observed in subsection \ref{sect-cmpFD}),
the CPU time comparison was made for a pure Dirichlet problem where both methods have approximately the same accuracy. However, if we compare cost versus accuracy in the Neumann problems the above observed speedups become twice, approximately, for the 2D case and even more for the 3D case.}

Finally, it is important to note that we do not guarantee that the chosen parameters and domain sizes are the true optimal ones and nothing more optimal can be found.

\begin{figure}[!h]
\begin{center}
\includegraphics[width=4cm]{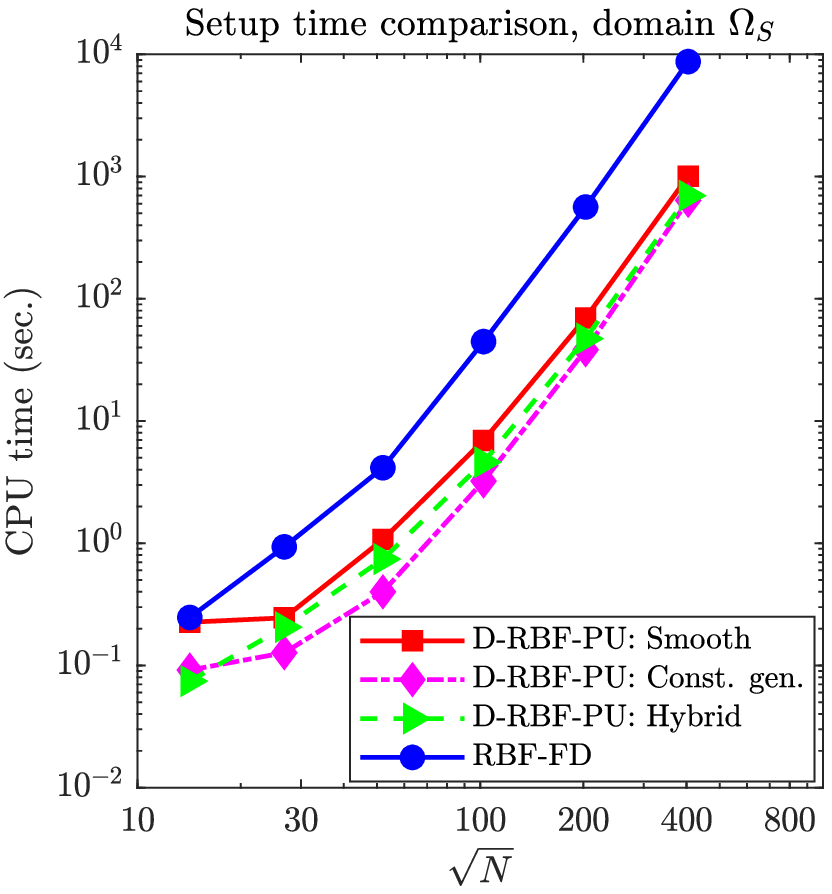}\includegraphics[width=4cm]{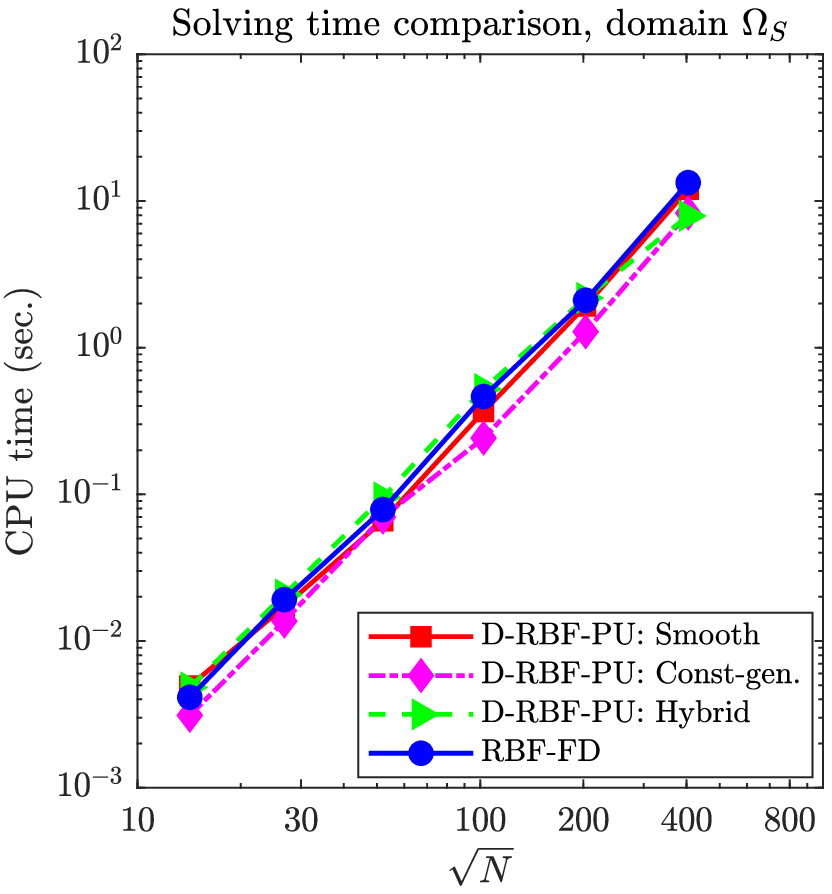}\\
\includegraphics[width=4cm]{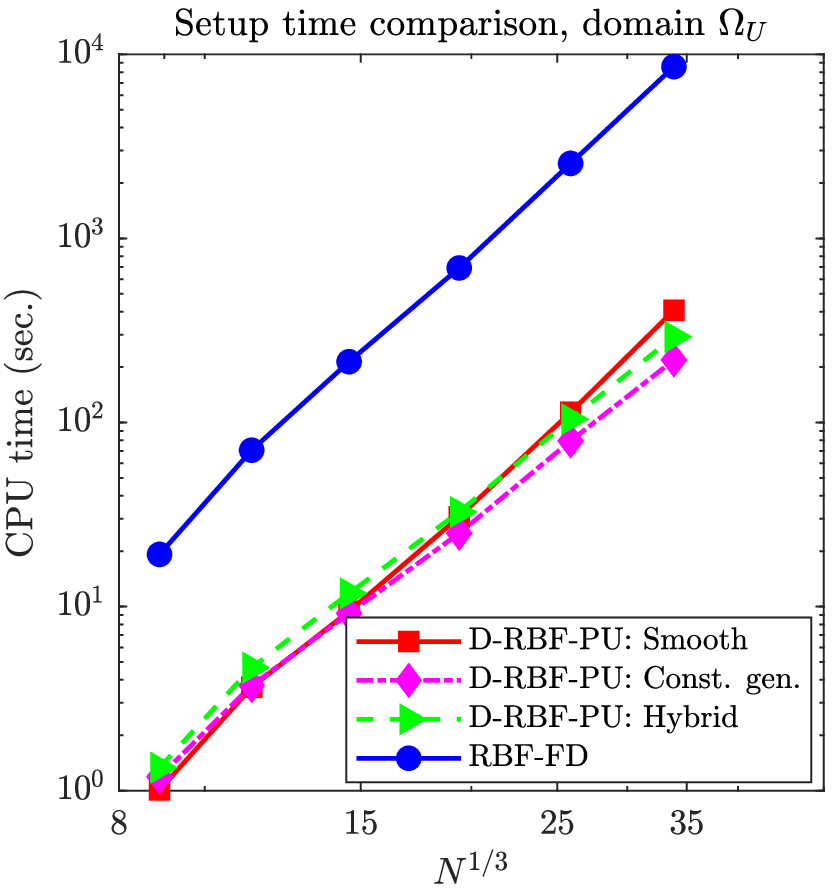}\includegraphics[width=4cm]{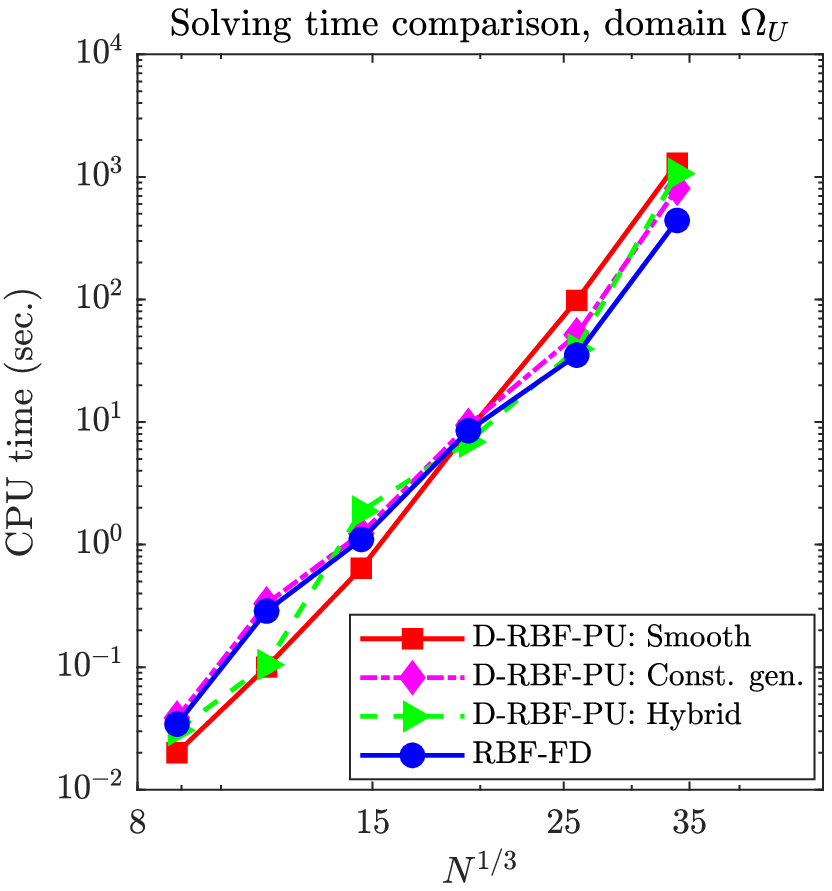}\includegraphics[width=4cm]{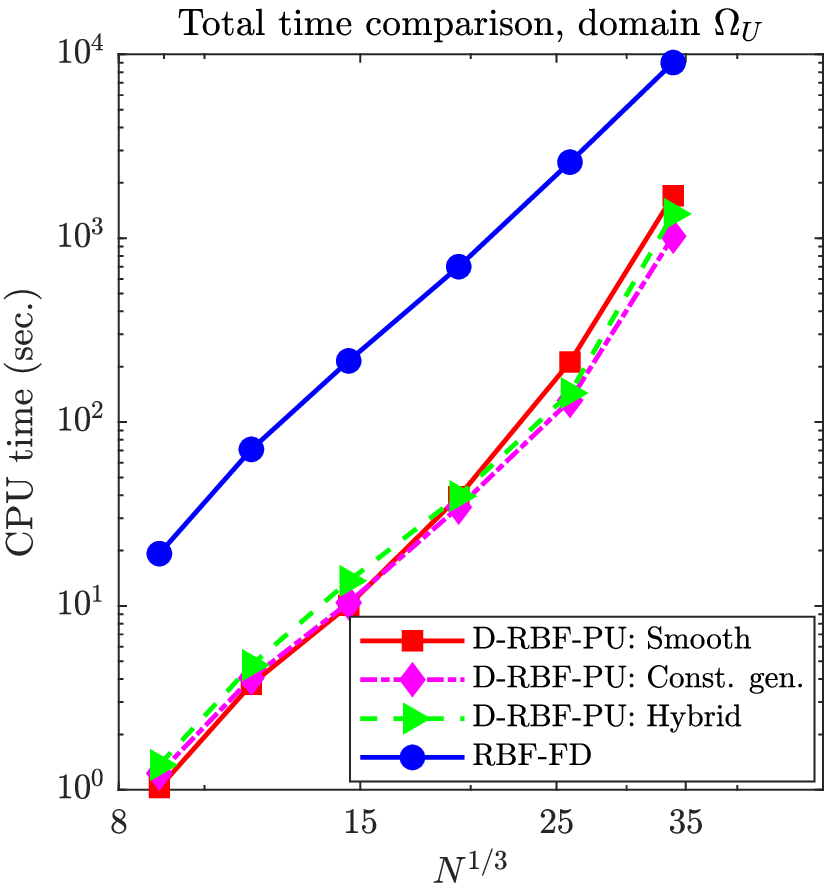}\\
\caption{\small{Comparison of the CPU times (sec.) between RBF-FD and D-RBF-PU methods for 2D (top) and $3D$ (down) problems.}}\label{fig_cpu}
\end{center}
\end{figure}

\section*{Conclusion}
The direct radial basis function partition of unity (D-RBF-PU) method is proposed for solving boundary and initial-boundary value problems.
The convergence properties and the stability issues are considered and some advantages of the new method are outlined.
{The advantage over the standard RBF-PU is that the new method avoids the action of PDE operators on PU weight functions.
This reduces both computational cost and algorithmic complexity without any significant influence on accuracy and stability.
More importantly, this provides a possibility to use some discontinuous PU weights allowing to have even more efficient schemes and recover the standard RBF-FD method as a special case.}

{In comparison with the RBF-FD, the new method needs to solve much fewer number of local linear systems for constructing the final differentiation matrix.
This reduces the computational costs, considerably.
In our experiments, average speedups of 5x with a smooth PU weight, 10x with a constant-generated PU weight and 9x with a hybrid PU weight are observed in both 2D and 3D examples.
Although for a pure Dirichlet problem both methods have approximately the same accuracy, the new method gives more accurate results for Neumann or Neumann-Dirichlet boundary value problems.}

Finally, we note that this method can be applied to a large class of PDE problems in engineering and sciences.
{In a follow-up work we will explore an application of the D-RBF-PU method for solving PDEs on surfaces.}

\section*{Acknowledgments}
I wish to express my deep gratitude to anonymous reviewers for their helpful comments
which improved the quality of the paper. Special thanks go to Behnam Hashemi (Shiraz University of
Technology) for his helpful comments about sparse matrix manipulation in \textsc{Matlab} and a careful proofreading.
Comments by Robert Schaback (Universit\"{a}t G\"{o}ttingen) about the approximation orders of the polyharmonic kernels are greatly acknowledged.


\end{document}